\documentclass[11pt,reqno]{amsart}
\usepackage{fullpage}
\usepackage{xcolor}
\usepackage[T1]{fontenc}                                   
\usepackage{amsfonts}
\usepackage[utf8]{inputenc}                               
\usepackage{comment}                                       
\usepackage{mparhack}                                      
\usepackage{amsmath,amssymb,amsthm,mathrsfs,eucal}                      

\usepackage{mathtools}
\usepackage{booktabs}

\usepackage{graphicx,subfig}                                                                
\usepackage{wrapfig}                                                                            
\usepackage[bookmarks=true,colorlinks=true]{hyperref}                      
\usepackage{bm}

\usepackage{dsfont}
\usepackage{esint}

\RequirePackage{color}
\RequirePackage{booktabs}
\RequirePackage{amsmath}
\usepackage{mathtools}
\usepackage{amsfonts}
\RequirePackage{etoolbox}

\allowdisplaybreaks[4]

\newcommand{\R}{\mathbb{R}}
\newcommand{\vertiii}[1]{{\left\vert\kern-0.25ex\left\vert\kern-0.25ex\left\vert #1
    \right\vert\kern-0.25ex\right\vert\kern-0.25ex\right\vert}}

\title{Kermack-McKendrick type models for epidemics with nonlocal aggregation terms}
\author{M. Di Francesco and F. Ghaderi Zefreh}
\date{}

\newtheorem{Theorem}{Theorem}[section]

\newtheorem{Definition*}{Definition}[section]
\newtheorem{Lemma}{Lemma}[section]
\newtheorem{Proposition}{Proposition}[section]
\newtheorem{Remark}{Remark}[section]

\begin{document}

\address{Marco Di Francesco - DISIM - Department of Information Engineering, Computer Science and Mathematics, University of L'Aquila, Via Vetoio 1 (Coppito)
67100 L'Aquila (AQ) - Italy}
\email{marco.difrancesco@univaq.it}

\address{Fatemeh Ghaderi Zefreh - DISIM - Department of Information Engineering, Computer Science and Mathematics, University of L'Aquila, Via Vetoio 1 (Coppito)
67100 L'Aquila (AQ) - Italy}
\email{fatemeh.ghaderizefreh@graduate.univaq.it}



\begin{abstract}
    We propose an approach to model spatial heterogeneity in SIR-type models for the spread of epidemics via \emph{nonlocal aggregation terms}. More precisely, we first consider an SIR model with spatial movements driven by nonlocal aggregation terms, in which the inter-compartment and intra-compartment interaction terms are distinct, and modelled through smooth interaction kernels. For the Cauchy problem of said model we provide a full well-posedness theory on $\R^2$ for $L^1\cap L^\infty \cap H^1$ initial conditions. The existence part is achieved by considering an approximated model with artificial linear diffusion, for which existence and uniqueness is proven via Duhamel's principle and Banach fixed point, and by providing suitable uniform estimates on the approximated solution in order to pass to the limit via classical compactness techniques. To prove uniqueness, we use classical $L^2$-stability which relies on the $H^1$-regularity of the solution. In the second part of the paper we provide a brief, general discussion on the steady states for these type of models, and display a specific example of non-trivial steady states for an SIS model with aggregations (driven by a single repulsive-attractive potential), the existence of which is determined by a threshold condition for a suitable "space-dependent" basic reproduction rate. We complement the analysis with numerical simulations on the SIS model.
\end{abstract}

\maketitle

\section{Introduction}
The mathematical modelling of the diffusion of infectious diseases is a classical topic in applied mathematics. The pioneering work of Kermack and McKendrick \cite{kermack} is usually referred to as the first important contribution to the subject. The key idea behind said classical model is that of dividing the human/animal population subject to a given infectious disease into several compartments such as \emph{susceptible} (the individuals who can catch the disease), \emph{infectious} (the individuals who can transmit the disease), \emph{recovered} (the individuals who have caught the disease and have recovered from it), and various other ones depending on the particular problem under study. Perhaps the most classical and most well-known version of this model is the so-called SIR, in which the three above mentioned compartments are the only ones considered, and which looks like
\begin{equation}\label{eq:SIR}
    \begin{cases}
        \dot{S}(t)=-\beta S(t)I(t),& \\
        \dot{I}(t)=\beta S(t)I(t) -\alpha I(t), & \\
        \dot{R}(t)=\alpha I(t).
    \end{cases}
\end{equation}
The quantity $S(t)$ (respectively $I(t)$ and $R(t)$) represents the total population of susceptible (respectively infectious and recovered) individuals. The quantities $S(t)$, $I(t)$, and $R(t)$ are functions of time $t\in [0,+\infty)$. They are assigned each one an initial condition at time $t=0$. The coefficients $\beta>0$ and $\alpha>0$ are usually called \emph{transmission rate} and \emph{recovery rate} respectively. In \eqref{eq:SIR}, only two transitions occur between two distinct compartments, namely from susceptible to infectious and from infectious to recovered. The total population $S(t)+I(t)+R(t)$ in \eqref{eq:SIR} is constant in time. More complex models include additional transitions (such as from $R$ to $S$ when there is loss of immunity) and additional compartments (such as quarantined, vaccinated, exposed, deceased, and so on). Providing a thorough list of contributions to the mathematical theory of SIR models and similar ones is quite challenging and goes beyond the aims of this paper. We refer to the book \cite{diekmann}, a classical one in this context, see also the review papers \cite{sattenspiel_review} and \cite{anita_capasso} and the more recent book \cite{martcheva}.

The set of approaches which start from models of the form \eqref{eq:SIR} and consider various variations depending on the problem under study has become a mathematical area of its own, which is usually referred to as "mathematical modelling in epidemiology", or simply \emph{mathematical epidemiology}. Apart from the aforementioned more complex models with additional transitions and compartments, other recurrent examples of variations from SIR-type models are the use of delay-in-time in the transition coefficients (see e. g. \cite{beretta}) and the introduction of \emph{structural} independent variables which have an impact on the involved parameters, the typical example being that of the so-called \emph{age structured} SIR type models. We mention in this framework the papers
\cite{iannelli_et_al,anita_iannelli,pugliese}. Other examples are multiple susceptible species (for example in animal-human transmissions (see \cite[Chapter 4]{martcheva}), and multiple strains (\cite[Chapter 8]{martcheva}).

\medskip
Starting from the pioneering \cite{thieme} (see also the more recent \cite{rass}), some researchers in this field started to include \emph{spatial heterogeneity} in the modelling of epidemics. Spatial heterogeneity may be seen in two main ways: either by assuming the movements of various "patches" (see e.g. \cite{lloyd_may,kuniya_et_al}), or by analysing the \emph{spatial spread} of "fronts" in the propagation of the epidemics. In this paper we deal with the latter approach. We introduce a two dimensional space variable $x\in \R^2$, which may vary in a given subset of $\R^2$. Having in mind the reference model \eqref{eq:SIR} as a starting point, the three population compartments $S$, $I$, and $R$ are now depending on space and time, that is $S=S(x,t)$, $I=I(x,t)$, and $R=R(x,t)$. Perhaps the simplest way to introduce space dependence in the model is by introducing space-dependent transition coefficients, see
\cite{takacs_horvath_farago,yang_et_al}.

Ideally, the introduction of spatial heterogeneity in the model should imply the formation of heterogeneous patterns (or clusters) or the formation of moving fronts. In a more general context, this is the typical behavior one may observe in "reaction-diffusion systems" as first observed in the pioneering work of A. Turing \cite{turing}. Based on that intuition, several authors have considered SIR models with diffusion on each species with different diffusion coefficients.
We mention the classical reference \cite{murray}. See also the more recent \cite{Li_2008,zhang_et_al,mohan_kumari}. In these models, linear diffusion terms are added to each equation in \eqref{eq:SIR} to model the spread of individuals. For example, the first equation in \eqref{eq:SIR} has an additional $d\Delta S$, with $d>0$ a diffusion constant (which is possibly distinct for each compartment) and with $\Delta$ being the two-dimensional Laplace operator.

Clearly, the recent COVID-19 pandemic has contributed to increase the attention on this subject, see \cite{zhi_et_al,mingione_et_al}. Reaction-diffusion versions of SIR type models have been also considered with age-structured part, see \cite{langlais_busenberg,kim}. Diffusion models with free boundaries have been studied in \cite{kim_lin_zhang}. Localised breaks and spike solutions were considered in \cite{gai_iron_kolokolnikov}. Other recent relevant results on SIR type models with diffusion are \cite{viguerie_et_al,grave_et_al}. The impact of road connections in the spread of the virus are considered in \cite{berestycki_et_al}. Applications to regional control are studied in \cite{capasso_anita}.
We refer to the review paper \cite{davydovych_et_al} for further references.

\medskip
An increasing attention has been recently devoted to describing the spread of individuals in epidemics modelling via \emph{nonlocal movements}. We mention in particular the research strand of \cite{kuniya_wang,yang_wang,bentout_et_al,li,cao_et_al}, all of which refer to epidemics modelling (except \cite{cao_et_al} devoted to Lotka-Volterra models) in which the Laplacian is replaced by its nonlocal verison $d(J\ast S - S)$, with $J\geq 0$ being an integrable kernel with unit mass on $\R^2$, and the symbol $\ast$ stands for convolution with respect to the spatial variables. The main idea behind this approach is that the use of the Laplacian in the diffusion term only takes into account of "local" interactions, that is, each individual only interacts with those at the same position, or with its nearest neighbours if we interpret linear diffusion as the macroscopic limit of microscopic random interactions. Describing space heterogeneity via nonlocal dispersal allows instead to consider interactions among individuals located at distances of macroscopic scale. We stress that both approaches typically consider \emph{diagonal} diffusion, that is, spatial interactions are limited to individuals of the same species.

In this paper, we propose an alternative approach to model nonlocality in spatial heterogeneity, based on so-called  \emph{nonlocal aggregation} terms in the model. We explain the basics of this approach to the non-expert reader. At a microscopic level, assuming for simplicity that we are dealing with $N$ individuals with time-varying positions $x_1(t),\ldots,x_N(t)\in \R^d$, with all individuals belonging to the same compartment (no transitions to other compartments), the velocity of each individual in a nonlocal aggregation terms is typically  expressed as a discrete convolution
\begin{equation}\label{eq:nonlocal_particle}
    \dot{x}_i(t)=-\frac{1}{N}\sum_{j=1}^N \nabla W (x_i(t)-x_j(t))\,,\qquad i=1,\ldots,N\,,
\end{equation}
where $W$ is a given \emph{interaction kernel}, or \emph{aggregation kernel}, which is typically assumed to be radial, $W(x)=w(|x|)$ for some one-dimensional profile $w:[0,+\infty)\rightarrow \R$. In \eqref{eq:nonlocal_particle}, each particle $x_i$ potentially interacts with any other particle $x_j$ with a "weight" depending on the relative position $x_i-x_j$ (the distance $|x_i-x_j|$ in case $W$ is radial). The profile $w$ models attractive interactions if $w'>0$, repulsive ones if $w'<0$. The use of such nonlocal interactions was introduced in the context of animal dispersal and swarming models, see \cite{mogilner,topaz}. It is also used in the modelling of crowd-dynamics, see e.g. \cite{colombo}, and in ecology \cite{cantrell}. Moreover, very popular models in cellular biology use a similar approach, see the well-known \cite{keller}.

The macroscopic PDE version of \eqref{eq:nonlocal_particle} is the nonlocal continuity equation
\begin{equation}\label{eq:nonlocal_intro}
    \partial_t \rho(x,t) - \mathrm{div}(\rho \nabla W\ast \rho)=0,
\end{equation}
in which the unknown function $\rho=\rho(x,t)$ represents the space-time depending density of individuals. In \eqref{eq:nonlocal_intro}, the discrete convolution of \eqref{eq:nonlocal_particle} is replaced by a continuous one. Depending on the applied context in mind, an additional, linear or nonlinear diffusion term $\Delta a(\rho)$ may be added on the right-hand side of \eqref{eq:nonlocal_intro}, in which case the equation \eqref{eq:nonlocal_intro} is called \emph{aggregation-diffusion} equation. Equations of the form \eqref{eq:nonlocal_intro} are very flexible tools to model complex space-dependent dynamics. Despite being relatively simple in their formulation, their long-time behavior may result into a variety of asymptotic states, ranging from homogenisation (convergence to a constant state) to the formation of a nontrivial pattern or to the emergence of blow-up in finite or infinite time. The latter two are often interpreted as the mathematical counterpart of \emph{self-organisation} phenomena in case the model arises from animal population dynamics. The study of the steady states of \eqref{eq:nonlocal_intro} is also of interest in flocking type models, see \cite{motsch_tadmor,dorsogna_et_al}.

The mathematical theory of aggregation diffusion equations has attracted great attention in the past twenty years. For the sake of brevity we mention here some of the main references for the existence theory. A classical functional analytical approach was used in \cite{bertozzi_carrillo_laurent, bertozzi_laurent}. The use of Wasserstein gradient flows developed in \cite{AGS} inspired the results in \cite{CDFLS,CMV}. Often, the interaction kernel $W$ features some singularities, for example in biological aggregation modelling and in population dynamics. Those singularities often result in the formation of concentrations in finite times, see the seminal result in \cite{JL}. This complicates the analysis of these models and motivates their study in the framework of measure solutions, see \cite{carrillo_choi_hauray, CDFLS}.

Applications e.g. in crowd movements \cite{degond}, multi-species populations \cite{CM}, and opinion formation \cite{during} motivated the study of systems of equations of the form \eqref{eq:nonlocal_intro} with more than one population species. In case of two species, with densities $\rho(x,t)$ and $\eta(x,t)$ respectively, a pretty general model is
\begin{equation}\label{eq:nonlocal_multi}
    \begin{cases}
    \partial_t \rho =\mathrm{div}(\rho\nabla (a(\rho,\eta)+W_{11}\ast \rho + W_{12}\ast \eta)), & \\
    \partial_t \eta =\mathrm{div}(\eta\nabla (b(\rho,\eta)+W_{21}\ast \rho + W_{22}\ast \eta)), &
    \end{cases}
\end{equation}
where we have included self-diffusion and cross-diffusion effects in $a$ and $b$. For the mathematical theory of these type of models we refer e.g. to \cite{difrafag,DiFEF}. We observe that in \eqref{eq:nonlocal_multi} there are four distinct interaction kernels $W_{ij}$, with $i,j\in\{1,2\}$. More precisely, the term $\nabla W_{ij}(x-y)$ models the extent in which an individual of the $i$-species located at $x$ is affected by an individual of the $j$-species located at $y$. The kernels $W_{11}$ and $W_{12}$ are called \emph{self-interaction} potentials, whereas $W_{12}$ and $W_{21}$ are called \emph{cross-interaction} potentials. Diversifying the nonlocal interaction terms depending on the species under consideration is reasonable in many applications. As a key example, one can consider predator-prey interactions, see \cite{difrafag2}, with $\rho$ and $\eta$ in \eqref{eq:nonlocal_multi} representing predator and prey species respectively, with $W_{12}$ attractive and $W_{21}$ repulsive, while $W_{11}$ and $W_{22}$ may be either zero or yielding the formation of a nontrivial aggregating pattern (which is the case, for example, if the self-interaction potentials are short-range repulsive and long-range attractive, see for example \cite{balague}). The degenerate case $a(\rho,\eta)=b(\rho,\eta)=(\rho+\eta)^2$ is particularly interesting (and challenging from the point of view of the existence of transient solutions) as it gives rise to sorting phenomena \cite{sorting} and to a variety of separated and non-separated steady states \cite{zoology}.

\medskip
The goal of this paper is to introduce a new approach in the study of Kermack-McKenrdick type models in which the spatial heterogeneity is expressed through nonlocal aggregation terms in the spirit of \eqref{eq:nonlocal_multi}. As a paradigmatic example, in the main part of our paper, we shall devote our attention to the SIR model with nonlocal aggregation terms
\begin{equation}\label{mainsystem_intro}
    \left\{\begin{array}{l}
\partial_t S=\operatorname{div}\left(S \nabla V_{S}[S, I, R]\right)-\beta S I, \\
\partial_t I=\operatorname{div}\left(I \nabla V_{I}[S, I, R]\right)+\beta S I-\alpha I, \\
\partial_t R=\operatorname{div}\left(R \nabla V_{R}[S, I, R]\right)+\alpha I,
\end{array}\right.
\end{equation}
where the unknown functions $S$, $I$, and $R$ depend on $x \in \mathbb{R}^{2}$ and $t>0$. Here, the operators $V_{S}$, $V_{I}$, and $V_{R}$ model nonlocal aggregation and are defined as
\begin{equation}\label{eq:nonlocal_terms}
\begin{aligned}
& V_{S}[S, I, R]=W_{S S}*S + W_{S I}*I + W_{S R}*R \\
& V_{I}[S, I, R]=W_{I S}*S + W_{I I}*I + W_{I R}*R \\
& V_{R}[S, I, R]=W_{R S}*S + W_{R I}*I + W_{R R}*R
\end{aligned}
\end{equation}
where the convolutions are taken only with respect to the space variable $x$. The constants $\alpha,\beta>0$ are assigned and justified as in \eqref{eq:SIR}.
The interaction kernels $W_{\xi\eta}$ with $\xi,\eta\in \{S,I.R\}$ are assigned. They model the way the individuals of each compartment adapt their velocity with respect to the distribution of the individuals of the three compartments. Such freedom in the choice of the interaction potentials may be useful in many ways. As an example, assuming that all potentials feature a short-range repulsion and a long-range attraction (in such a way to possibly yield formation of nontrivial stationary patterns in the limit), one may consider a wider repulsive range in the potential $W_{SI}$ compared to the one of $W_{SS}$, due to the fact that susceptible individuals are keener to interact with other susceptible individuals rather than with infectious ones. In this paper we assume for simplicity that all kernels satisfy the regularity assumption
\begin{align}
    & D^\gamma W_{\xi\eta}\in L^1(\R^2)\cap L^\infty(\R^2)\nonumber\\
    & \hbox{for all multi-indexes $\gamma$ with $1\leq |\gamma|\leq 2$ and for all $\xi,\eta\in \{S,I,R\}$}.\label{eq:ass_reg_W}
\end{align}
Moreover, we assume
\begin{equation}\label{eq:ass_sym_W}
    W_{\xi\eta}(-x)=W_{\xi\eta}(x)\qquad \hbox{for all $x\in \R^2$ and for all $\xi,\eta\in \{S,I,R\}$}\,.
\end{equation}
The assumption \eqref{eq:ass_reg_W} is merely technical in that it eases the mathematical theory of \eqref{mainsystem_intro}. More general cases might be considered in future studies. Assumption  \eqref{eq:ass_sym_W} relaxes the radial condition $W_{\xi\eta}(x)=w_{\xi\eta}(|x|)$ for some one-dimensional profile $w_{\xi\eta}$, which is more often used in this context. The problem we study is a plain Cauchy problem with $x\in \R^2$, with no boundary condition, for which we provide a full well-posedness theory.

The second purpose of this work is to test our nonlocal transport approach on a classical model featuring multiple steady states, the stability of which is determined by the size of a so-called \emph{basic reproduction number} $\mathcal{R}_0$. To this aim, in the last two sections we study the one-dimensional nonlocal SIS model
\begin{equation}\label{eq:SIS_space_intro}
    \begin{dcases}
    \partial_t S = \partial_x (S W'*(S+I)) -\beta SI +\alpha I & \\
      \partial_t I = \partial_x (I W'*(S+I)) +\beta SI -\alpha I &
    \end{dcases}
\end{equation}
with the special choice for the interaction kernel
\begin{equation}\label{eq:W_special}
    W(x)= x^2-\gamma|x|\,,\qquad \gamma >0.
\end{equation}
System \eqref{eq:SIS_space_intro} is a two-compartment model in which the transition from infectious to recovered is replaced by a transition from infectious to susceptible, to express loss of immunity upon recovery. The choice of $W$ yields the emergence of a nontrivial pattern for the total population. We provide existence of space-inhomogeneous disease-free and endemic equilibrium for \eqref{eq:SIS_space_intro} under suitable assumptions. Moreover, we provide numerical simulations suggesting that these steady states are stable.

\medskip
Before explaining our results in detail, we provide here a further digression on the existing literature with the twofold goal of motivating the use of our model \eqref{mainsystem_intro} and of providing possible further directions for our study.
\begin{enumerate}
    \item Our approach in using nonlocal terms to model spatial heterogeneity has the potential to be proposed at the level of \emph{moving agents}, in the spirit of the ODE system \eqref{eq:nonlocal_particle}. This remark is possibly of interest with respect to the direction of the research in
\cite{sattenspiel_et_al}, which aims at simulating the spread of an epidemics via an agent-based approach.
\item One of our goals in proposing this approach is to analyse the impact of nonlocal transport in the qualitative properties of SIR-type models, including the long time behavior. As already observed in \cite{colombo_garavello_marcellini_rossi} (in which spatial heterogeneity is modelled via linear transport terms), spatial movements may impact very highly on the behavior of the solution. This opens possible new directions in the comparison of our model with more classical models (space homogeneous or with linear diffusion), not to mention the possibility of calibrating the model with available datasets, of making the model more realistic adding further compartments, and of using additional containment terms in the spirit of control theory. It is worth mentioning at this stage that these tasks have been tackled in \cite{albi_et_al} in the context of Covid-19 modelling, with a slightly different approach based on kinetic theory, in which transport has been simplified to include bidirectional terms. For the kinetic approach to Covid-19 modelling see also
\cite{bertaglia_pareschi_toscani} and
\cite{dimarco_perthame_toscani_zanella}.
\item We mention here the paper \cite{nature}, which seems perhaps as close as it gets to our approach. In said paper, a nonlocal transport version of SIR is formulated via the so called dynamic density functional theory, a modelling technique which is borrowed from soft matter theoretical physics. The simulations in \cite{nature} provide evidence of phase transition phenomena and the formation of colonies. We observe that our model \eqref{mainsystem_intro} contains cross-interaction terms which are not included in \cite{nature}.
\item The literature of nonlocal interaction equations and systems provides a fairly rich catalogue of possible behaviors for large times, even if we confine our analysis to smooth interaction potentials. The case of attractive potentials produces the formation of delta steady states in infinite times (see \cite{BuDiF}). The case of potentials which are repulsive for short ranges and attractive for large ranges implies the formation of multiple clusters, see \cite{raoul}. In the case of multiple species, we mention the result in \cite{difrafag2}, in which the emergence (for large times) of multiple clusters for both species is proven. The case of potentials with singular repulsion raises various phenomena related with dimensionality of the clusters, see \cite{balague}. An interesting further direction will be to study (both analytically and numerically) the competition between said effect, which tends to drive the solutions towards heterogeneous states, and the effect of the SIR-type terms, which typically produce homogeneous steady states for large times. Our result in Section \ref{sec:steady} below is a first attempt in this direction, with a very specific choice of the interaction potentials.
\end{enumerate}

\medskip
As mentioned above, our prior task is to provide a mathematical well-posedness theory for the Cauchy problem for \eqref{mainsystem_intro} under fairly general assumptions on the initial conditions. The main mathematical challenge behind this problem consists of the simultaneous presence of nonlocal transport terms and of reaction terms. The nonlocal reaction terms are typically dealt with via characteristics (with a fixed point strategy) or via optimal transport (JKO scheme \cite{JKO}). The latter cannot be used here due to the presence of the reaction terms, some of which are nonlinear. The former also has some inconveniences when it comes to proving global properties, especially if one wants to prove propagation of a certain extent of regularity. Therefore, we use the following approach:
\begin{itemize}
    \item we regularise the model with artificial linear diffusion terms, for which we prove local and global existence and uniqueness of solutions;
    \item we prove uniform estimates on the solution with respect to the artificial diffusion; we stress that the nonlinear reaction terms require strong compactness at least in some $L^p$ space;
    \item we obtain the existence of weak solutions to \eqref{mainsystem_intro} by letting the diffusion go to zero;
    \item we finally prove uniqueness in $L^2$.
\end{itemize}
All the above may be summarised in the following existence and uniqueness result, which is the main result of the paper and it is therefore stated here for simplicity.

\begin{Theorem}[Existence and uniqueness of solutions]\label{thm:main}
    Assume the interaction kernels $W_{\xi \eta}$ satisfy \eqref{eq:ass_reg_W} and \eqref{eq:ass_sym_W}. Assume the initial data $S_0, I_0, R_0$ belong to $L^1(\R^2)\cap L^\infty(\R^2) \cap H^1(\R^2)$ and are nonnegative. Then, for any arbitrary time $T\geq 0$, there exists one and only one
    \[
    (S,I,R)\in L^\infty([0,T]\,;\,\, (L^1(\R^2)\cap L^\infty(\R^2)\cap H^1(\R^2)))^3
    \]
    solving \eqref{mainsystem_intro} in the weak sense and having $(S_0, I_0, R_0)$ as initial datum.
\end{Theorem}

We mention here that we expect the techniques used to prove Theorem \ref{thm:main} to be easily extended to more general nonlocal aggregation systems with many species and with reaction terms featuring linear growth at infinity (assuming the regularity of \eqref{eq:ass_reg_W}). The main difficulty in the case of more general reaction parts lies in performing estimates that guarantee the solution to be global-in-time. The SIR type model we consider in this paper is quite special, since it allows us to perform such estimates taking advantage of the nonnegativity of solutions and the particular form of the reaction part of the species $S$. The extension of this result to more complicated systems is left to future studies.

As mentioned above, the last part of the paper is devoted to steady states for \eqref{eq:SIS_space_intro}, seen as a paradigmatic example to show that the approach we propose in this paper may produce asymptotic states which slightly correct classical homogeneous steady states. More precisely, we shall prove the existence of inhomogeneous stationary states for the system \eqref{eq:SIS_space_intro} with $W$ given in \eqref{eq:W_special}. We will show that a space dependent basic reproduction rate $\mathcal{R}_0$ (as in the classical SIR-type models literature) may be defined to determine existence or non existence of endemic equilibria. In our example, the definition of $\mathcal{R}_0$ involves a parameter from the aggregation kernels (namely, the size of the repulsive range) on top of the usual ones used in classical models, such as the transmission rate and the recovery rate. In the last section we simulate time-depending solutions to \eqref{eq:SIS_space_intro} approaching respectively the disease free equilibrium and the endemic one. The simulation of the latter seems to suggest possible phase transitions between the two steady states. We will devote a further study to that.

\medskip
The paper is structured as follows. In section \ref{sec:approximated} we add an artificial linear diffusion term to \eqref{mainsystem_intro} and prove local existence and uniqueness via Banach fixed point (subsection \ref{subsec:local}), nonnegativity and total mass conservation (subsection \ref{subsec:positivity}), and global existence (subsection \ref{subsec:global}). In section \ref{sec:existence} we first provide additional uniform estimates of the $H^1$ norm (subsection \ref{subsec:H1}) and then prove that the vanishing diffusion limit converge a solution to our model \eqref{mainsystem_intro}  (subsection \ref{subsec:vanishing}), which completes the existence proof. Uniqueness is proven in section \ref{sec:unique}. Section \ref{sec:steady} is devoted to steady states. In subsection \ref{subsec:discussion} we provide a general discussion on the existence of steady states, whereas in subsection \ref{subsec:example} we provide a specific example for an SIS model with repulsive-attractive interactions. the main result of this section is Theorem \ref{thm:states}. Section \ref{sec:numerical} is devoted to numerical simulations for \eqref{eq:SIS_space_intro}.

\section{The approximated problem}\label{sec:approximated}

\subsection{Local existence and uniqueness for the approximated problem}\label{subsec:local}

We consider, for a fixed $\varepsilon>0$, the approximated problem
\begin{equation}\label{mainsystem}
    \left\{\begin{array}{l}
S_{t}=\operatorname{div}\left(S \nabla V_{S}[S, I, R]\right)-\beta S I+\varepsilon \Delta S, \\
I_{t}=\operatorname{div}\left(I \nabla V_{I}[S, I, R]\right)+\beta S I-\alpha I+\varepsilon \Delta I, \\
R_{t}=\operatorname{div}\left(R \nabla V_{R}[S, I, R]\right)+\alpha I+\varepsilon \Delta R,
\end{array}\right.\,
\end{equation}
where the nonlocal operators $V_S$, $V_I$, and $V_R$ are defined in \eqref{eq:nonlocal_terms} and where $\varepsilon>0$ is fixed. We equip \eqref{mainsystem} with initial conditions
$S(x,0)=S_0(x)$, $I(x,0)=I_0(x)$, $R(x,0)=R_0(x)$. For future use, we recall the definition of the two-dimensional Gaussian kernel
\begin{equation}\label{eq:gaussian}
    G_\varepsilon(x,t)=\frac{1}{4\pi \varepsilon t}e^{-|x|^2/4\varepsilon t}\,.
\end{equation}
We recall that
\begin{equation}\label{eq:gaussian_mass}
\|G_\varepsilon(\cdot,t)\|_{L^1(\R^2)}=1 \qquad \hbox{for all $t>0$.}
\end{equation}
Moreover, there exists a constant $C>0$ depending on $\varepsilon$ but independent of $t$, such that
\begin{equation}\label{eq:gaussian_grad}
    \left\|\nabla G_{\varepsilon}(t)\right\|_{L^{1}(\R^2)}=C t^{-1 / 2}\qquad \hbox{for all $t>0$.}
\end{equation}

\begin{Remark}[Use of generic constants]
    \emph{In this section, we will often make use of generic constants depending on the structural parameters of the model \eqref{mainsystem}. Since we are dealing with finitely many of those constants, we will use the same notation $C\geq 0$ for such generic constants. We stress, however, that in this subsection and in the next one, this constant may depend on $\varepsilon$. This is not a problem as long as we have to prove existence and uniqueness for the approximated problem \eqref{mainsystem}. In subsection \ref{subsec:global}, we will instead need to prove uniform estimates with respect to $\varepsilon$. In that case, our generic constant will be independent of $\varepsilon$. The same will happen in section \ref{sec:existence}. In some situations, the constant $C$ will also depend on a finite time horizon $T$ and/or on the initial data. We will make this clear later on when needed.}
\end{Remark}

For future use, we state the following Lemma.

\begin{Lemma}\label{lem:nonlocal_etsimates}
    Let $T\geq0$. Assume $W_{\xi \eta}$ satisfies \eqref{eq:ass_reg_W} for all $\xi,\eta\in \{S,I,R\}$. Then, there exists a constant $C$ depending only on the interaction kernels $W_{\xi\eta}$, $\xi,\eta\in \{S,I,R\}$, such that, for all two-dimensional multi-indexes $\gamma$ with $1\leq|\gamma|\leq 2$, there holds
    \begin{align}
       & \sup_{0\leq t\leq T}\left[\|D^\gamma V_S[S,I,R]\|_{L^\infty(\R^2)}+\|D^\gamma V_I[S,I,R]\|_{L^\infty(\R^2)}+\|D^\gamma V_R[S,I,R]\|_{L^\infty(\R^2)}\right]\nonumber\\
       & \ \leq C \sup_{0\leq t\leq T}\left(\|S\|_{L^p(\R^2)}+\|I\|_{L^p(\R^2)}+\|R\|_{L^p(\R^2)}\right)\label{eq:est_kernels},
    \end{align}
    for all $p\in[1,+\infty]$. Moreover, for all multi-indexes $\gamma$ such that $|\gamma|=3$, there holds
      \begin{align}
       & \sup_{0\leq t\leq T}\left[\|D^\gamma V_S[S,I,R]\|_{L^2(\R^2)}+\|D^\gamma V_I[S,I,R]\|_{L^2(\R^2)}+\|D^\gamma V_R[S,I,R]\|_{L^2(\R^2)}\right]\nonumber\\
       & \ \leq C \sup_{0\leq t\leq T}\left(\|\nabla S\|_{L^p(\R^2)}+\|\nabla I\|_{L^p(\R^2)}+\|\nabla R\|_{L^p(\R^2)}\right),\label{eq:est_kernels2}
    \end{align}
    for all $p\in [1,2]$.
\end{Lemma}

\begin{proof}
    The proof follows easily by differentiating the interaction kernels in the convolution operators $V_S, V_I, V_R$ and by applying Young's inequality for convolutions. Recall that assumptions \eqref{eq:ass_reg_W} implies $D^\gamma W_{\xi,\eta}$ belong to $L^p(\R^2)$ for all $p\in [1,+\infty]$ by $L^p$-interpolation for all $1\leq |\gamma|\leq 2$.
\end{proof}

We start by proving the local-in-time existence of solutions for \eqref{mainsystem} via Duhamel's principle and Banach fixed point.

\begin{Proposition}\label{prop:local}
Assume \eqref{eq:ass_reg_W} and let the initial condition satisfy $\left(S_{0}, I_{0}, R_{0}\right) \in\left(L^{1} \cap L^{\infty}\left(\mathbb{R}^{2}\right)\right)^{3}$. Then, there exists $T> 0$ such that there exists one and only one classical solution $(S,I,R) \in (C^{1,2}_{t,x}(\mathbb R^2 \times [0,T]))^3$ to system \eqref{mainsystem} in $\mathbb R^2 \times [0,T]$ with initial condition $\left(S_{0}, I_{0}, R_{0}\right)$.
\end{Proposition}
\begin{proof}
For $T>0$ and $r>0$ we define the Banach space $\left(X_{r, T},\vertiii{\cdot}_T\right)$ as
\begin{align*}
    & X_{r, T}:=\left\{\vphantom{\int}(S, I, R) \in C\left([0, T] ; \left(L^{1}\left(\mathbb{R}^{2}\right) \cap L^{\infty}\left(\mathbb{R}^{2}\right)\right)^3\right):\,\vertiii{S}_{T}+\vertiii{I}_{T}+\vertiii{R}_{T}\leq r\right\}\,,
\end{align*}
where, for any continuous curve $f:[0,T]\rightarrow L^1(\R^2)\cap L^\infty(\R^d)$, we have denoted the norm
\[
\vertiii{f}_T=\sup_{0\leq t\leq T}  \left(\|f(\cdot,t)\|_{L^\infty(\R^d)}+\|f(\cdot,t)\|_{L^1(\R^d)}\right)
\]
and by abuse of notation, we also denote
\[\vertiii{(S,I,R)}_T=\vertiii{S}_{T}+\vertiii{I}_{T}+\vertiii{R}_{T}\,.\]
Our goal is to define a map $\mathcal{T}: X_{r, T} \rightarrow X_{r, T}$ such that
$$
\begin{gathered}
\mathcal{T} (\tilde{S},\Tilde{I},\Tilde{R})(\cdot,t):=G_{\varepsilon}(\cdot,t)*\left(\begin{array}{cc}
S_0\\ I_0\\ R_0 \end{array} \right)\\
+ \int_{0}^{t} \left(\begin{array}{cc}
 G_{\varepsilon}(.,t-s)*(\mathrm{div}(\tilde{S} \nabla V_S [\Tilde{S},\Tilde{I},\Tilde{R}])-\beta \Tilde{S} \Tilde{I})(.,s)\\
 G_{\varepsilon}(.,t-s)*(\mathrm{div}(\tilde{I} \nabla V_I [\Tilde{S},\Tilde{I},\Tilde{R}])+\beta \Tilde{S} \Tilde{I}-\alpha \tilde{I})(.,s)\\
 G_{\varepsilon}(.,t-s)*(\mathrm{div}(\tilde{R} \nabla V_R [\Tilde{S},\Tilde{I},\Tilde{R}])+\alpha \Tilde{I})(.,s)\\
 \end{array}\right) ds.
\end{gathered}
$$
If we can prove that $\mathcal{T}$ is well-defined and a contraction for some $T>0$ and $r>0$, then by Banach's fixed point theorem we can conclude that there exists one and only one fixed point for $\mathcal{T}$. Then, a standard Duhamel's principle for the linear diffusion equation implies that such a fixed point is the unique classical solution to system \eqref{mainsystem} for all $0 \leq t \leq T$. We shall use the notation
\[\mathcal{T} (\tilde{S},\Tilde{I},\Tilde{R})=\left(\mathcal{T}_1(\tilde{S},\Tilde{I},\Tilde{R}), \mathcal{T}_2(\tilde{S},\Tilde{I},\Tilde{R}), \mathcal{T}_3(\tilde{S},\Tilde{I},\Tilde{R})\right)\,.\]
In order to prove that $\mathcal{T}$ is well-defined, we shall prove that we can find $r,T>0$ such that, if $(\tilde{S}, \tilde{I}, \tilde{R}) \in X_{r, T}$, then $\mathcal{T}(\tilde{S}, \tilde{I}, \tilde{R}) \in X_{r, T}$. In other words, we need to prove the assumption $(\tilde{S}, \tilde{I}, \tilde{R}) \in X_{r, T}$ implies the estimate $\vertiii{\mathcal{T}_{1}(\tilde{S}, \tilde{I}, \tilde{R})}_{T}+\vertiii{\mathcal{T}_{2}(\tilde{S}, \tilde{I}, \tilde{R})}_{T}+ \vertiii{\mathcal{T}_{3}(\tilde{S}, \tilde{I}, \tilde{R})}_{T}\leq r$. We start by estimating $L^p$ norms, $p\in\{1,+\infty\}$, for fixed positive times $t>0$ (we omit the time dependency for simplicity in the notation) as follows:

\begin{align*}
   & \|\mathcal{T}_1 (\tilde{S},\Tilde{I},\Tilde{R})\|_{L^{p}(\mathbb{R}^{2})} \leq \|\mathcal{I}_1\|_{L^{p}(\mathbb{R}^{2})}+\|\mathcal{I}_2\|_{L^{p}(\mathbb{R}^{2})},\\
    & \mathcal{I}_1= G_\varepsilon *S_0,\\
    &\mathcal{I}_2 = \int_{0}^{t} G_\varepsilon(.,t-s)*\left(\mathrm{div}(\tilde{S}\nabla V_S [\tilde{S},\tilde{I},\tilde{R}]) -\beta \tilde{S}\tilde{I}\right)(.,s)ds\,.
\end{align*}
We control the norm of $\mathcal{I}_1$ via Young's inequality for convolutions,
\begin{align*}
   &  \|\mathcal{I}_1\|_{L^{p}(\mathbb{R}^{2})}=\left\|G_{\varepsilon} * S_{0}\right\|_{L^{p
   }\left(\mathbb{R}^{2}\right)} \leq \left\|G_{\varepsilon}\right\|_{L^{1}\left(\mathbb{R}^{2}\right)}\left\|S_{0}\right\|_{L^{p}\left(\mathbb{R}^{2}\right)} = \left\|S_{0}\right\|_{L^{p}\left(\mathbb{R}^{2}\right)}\,.
\end{align*}
As for $\mathcal{I}_2$, we have the estimate
\begin{align*}
    & \left| \mathcal{I}_2\right| \leq \beta \int_{0}^{t} G_{\varepsilon}(., t-s) *|\tilde{S} \tilde{I}|(\cdot,s) d s + \int_{0}^{t} \left|\nabla G_{\varepsilon}(., t-s) *\left(\tilde{S} \nabla V_{S}[\tilde{S}, \tilde{I}, \tilde{R}]\right)(\cdot,s)\right| d s\,,
 \end{align*}
 where the last convolution above is computed by taking the scalar product of the two vector fields $\nabla G_\varepsilon$ and $\tilde{S}\nabla V_S$. The above implies, once again via Young's inequality for convolutions,
 \begin{align*}
& \|\mathcal{I}_2\|_{L^{p}(\mathbb{R}^{2})} \leq \beta \int_0^t \left \| G_{\varepsilon}(.,t-s) \right\|_{L^{1}(\mathbb{R}^2)} \left\|\tilde{S}(\cdot,s)\right\|_{L^{\infty}(\mathbb{R}^2)} \left\| \tilde{I}(\cdot,s) \right\|_{L^{p}(\mathbb{R}^2)}ds \\
 & +\int _0 ^t \left\| \nabla G_{\varepsilon}(.,t-s) \right\|_{L^{1}(\mathbb{R}^2)}\left \| \tilde{S}(\cdot,s) \right\|_{L^{p}(\mathbb{R}^2)} \left\| \nabla V_S[\tilde{S},\tilde{I},\tilde{R}](\cdot,s) \right\|_{L^{\infty}(\mathbb{R}^2)} ds.
\end{align*}
Using \eqref{eq:est_kernels}, we have
\begin{align*}
    & \left\|\nabla V_{S}[\tilde{S}, \tilde{I}, \tilde{R}]\right\|_{L^{\infty}\left(\mathbb{R}^{2}\right)}
\leq C\sup_{0\leq t\leq T}\left(\|\tilde{S}\|_{L^1(\R^2)}+\|\tilde{I}\|_{L^1(\R^2)}+\|\tilde{R}\|_{L^1(\R^2)}\right)\,,
\end{align*}
and the assumption $(\tilde{S},\tilde{I},\tilde{R})\in X_{r,T}$ implies there exists $C\geq 0$ such that
\[
\left\|\nabla V_{S}[\tilde{S}, \tilde{I}, \tilde{R}]\right\|_{L^{\infty}\left(\mathbb{R}^{2}\right)}
\leq Cr\,.
\]
Hence, there exist a constants $C\geq 0$ independent of $\varepsilon$ and $t$ such that
\begin{align*}
    & \left\|\mathcal{T}_1 (\tilde{S}, \tilde{I}, \tilde{R})\right\|_{L^{p}\left(\mathbb{R}^{2}\right)} \leq\left\|S_{0}\right\|_{L^{p}(\R^2)}\\
& \ +\beta \int_{0}^{t}\left\|\tilde{S}(\cdot,s)\right\|_{L^{\infty}\left(\R^{2}\right)}\left\|\tilde{I}(\cdot,s)\right\|_{L^{p}\left(\mathbb{R}^{2}\right)} d s+C r\int_{0}^{t} (t-s)^{- \frac{1}{2}} \left\| \tilde{S}(\cdot,s) \right\|_{L^{p}\left(\mathbb{R}^{2}\right)} d s \\
& \ \leq C\left( r^{2} t+ r^{2} t^{\frac{1}{2}}+1\right)\,.
\end{align*}
Here, we have also used \eqref{eq:gaussian_grad}. We observe that $C$ may depend on the initial datum $S_0$ and on structural parameters of the problem \eqref{mainsystem_intro} such as $\alpha$, $\beta$, $W_{\xi\eta}$ for $\xi,\eta\in \{S,I,R\}$.
We may now take the supremum with respect to $t\in [0,T]$ and get, for a suitable $C$,
\begin{equation}\label{eq:estimate0}
\vertiii{\mathcal{T}_1 (\tilde{S}, \tilde{I}, \tilde{R})}_{T} \leq C\left( T^{\frac{1}{2}} r^{2} + r^{2} T+1\right) \,.
\end{equation}
The estimate \eqref{eq:estimate0} may be easily extended to the two components $\mathcal{T}_2 (\tilde{S}, \tilde{I}, \tilde{R})$ and $\mathcal{T}_3 (\tilde{S}, \tilde{I}, \tilde{R})$. The additional reaction terms are controlled in a very similar way as above. We remark that some of these terms may imply linear growth with respect to $r$ on the right-hand side. We therefore conclude there exists a constant $C\geq 0$ independent of $\varepsilon$ and of $T$ such that
\begin{equation}\label{eq:estimate1}
    \vertiii{\mathcal{T}(\tilde{S},\tilde{I},\tilde{R})}_T\leq C\left(1+ (r+r^2)(T^{1/2}+T)\right)\,.
\end{equation}

We next prove the contraction estimate for $\mathcal{T}$ on $X_{r,T}$. Let
$\left(S_1,I_1,R_1\right),\left(S_2,I_2,R_2\right)\in X_{r,T}$.
We have
\begin{align*}
& \vertiii{\mathcal{T}_1\left(S_1,I_1,R_1\right)-\mathcal{T}_1\left(S_2,I_2,R_2\right)}_T\\
& \ = \sup_{0\leq t\leq T}\left(\left\|\mathcal{T}_1\left(S_1,I_1,R_1\right)-\mathcal{T}_1\left(S_2,I_2,R_2\right)\right\|_{L^{\infty}\left(\mathbb{R}^2\right)}+\left\|\mathcal{T}_1 \left(S_1,I_1,R_1\right)-\mathcal{T}_1\left(S_2,I_2,R_2\right)\right\|_{L^{1}\left(\mathbb{R}^2\right)}\right).
\end{align*}

Below we provide contraction in $L^p$ for $p\in\{1,+\infty\}$ by using Young's inequality for convolution and estimating
\begin{align*}
    & \left\|\mathcal{T}_1\left(S_1,I_1,R_1\right)-\mathcal{T}_1\left(S_2,I_2,R_2\right)\right\|_{L^p\left(\mathbb{R}^2\right)}
 \leq \beta \int_{0}^{t}\left\|G_{\varepsilon}(.,t-s)\right\|_{L^{1}}\left\|S_{1} I_{1}(\cdot,s)-S_{2} I_{2}(\cdot,s)\right\|_{L^{p}} d s\\
& \ +\int_{0}^{t}\|\nabla G(.,t-s)\|_{L^{1}}\left\|S_{1} \nabla V_{S}\left[S_{1} I_{1}, R_{1}\right](\cdot,s)-S_{2} \nabla V_S \left[S_{2} I_{2}, R_{2}\right](\cdot,s)\right\|_{L^{p}} d s\,.
\end{align*}
Since
\begin{align*}
    & \left\|S_{1} I_{1}-S_{2} I_{2}\right\|_{L^{p}}=\left\|S_{1} I_{1}-S_{1} I_{2}+S_{1} I_{2}-S_{2} I_{2}\right\|_{L^{p}}\\
& \ \leq\left\|S_{1}\right\|_{L^{\infty}}\left\|I_{1}-I_{2}\right\|_{L^{p}}+\left\|I_{2}\right\|_{L^{\infty}}\left\|S_{1}-S_{2}\right\|_{L^{p}},
\end{align*}
and
\begin{align*}
    & \left\| S_{1} \nabla V_{S}\left[S_{1}, I_{1}, R_{1}\right]-S_2\nabla V_{S}\left[S_{2} I_{2}, R_{2}\right]\right\|_{L^{p}}\\
& \ \leq\left\| S_{1}\right\|_{L^{p}}\left\| \nabla V_{S}\left[S_{1}, I_{1}, R_{1}\right]-\nabla V_{S}\left[S_{2}, I_{2}, R_{2}\right]\right\|_{L^{\infty}} +\left\| S_{1}-S_{2}\right\|_{L^{p}}\left\| \nabla V_{S}\left[S_{2}, I_{2}, R_{2}\right] \right\|_{L^{\infty}},
\end{align*}
we may estimate the terms involving the nonlocal operator $\nabla V_S$ via \eqref{eq:est_kernels}, namely
\begin{align*}
    & \left \|\nabla V_S\left[S_2,I_2,R_2\right]\right\|_{L^{\infty}} \leq C\sup_{0\leq t\leq T}\left(\left\|S_2\right\|_{L^1} + \left\|I_2\right\|_{L^1} + \left\|R_2\right\|_{L^1}\right) \leq Cr,
\end{align*}
and similarly, for a suitable constant $C\geq 0$,
\begin{align*}
    & \left \|\nabla V_S\left[S_1,I_1,R_1\right] - \nabla V_S\left[S_2,I_2,R_2\right]\right\|_{L^{\infty}}
\leq C \sup_{0\leq t\leq T}\left(\left\|S_1-S_2\right\|_{L^{1}}+  \left\|I_1-I_2\right\|_{L^{1}} +  \left\|R_1-R_2\right\|_{L^{1}}\right) \,,
\end{align*}
to obtain, once again using \eqref{eq:gaussian_grad} and for a suitable constant $C$,
\begin{align*}
    & \left\|\mathcal{T}_1\left(S_1,I_1,R_1\right)-\mathcal{T}_1\left(S_2,I_2,R_2\right)\right\|_{L^{p}\left(\mathbb{R}^2\right)}\\
    & \ \leq C t \left (\left\|S_1\right\|_{L^{\infty}}\left\|I_1-I_2\right\|_{L^{p}}+ \left\|I_2\right\|_{L^{\infty}}\left\|S_1-S_2\right\|_{L^{p}}\right)\\
    & \
    + C t^{\frac{1}{2}}\left(r\left\|S_1-S_2\right\|_{L^{p}}+\left\|S_1\right\|_{L^{p}}\left(\vertiii{S_1-S_2}_{T} + \vertiii{I_1-I_2}_{T} + \vertiii{R_1-R_2}_{T}\right)\right)\\
    & \ \leq C rt\left(\vertiii{I_1-I_2}_{T} + \vertiii{S_1-S_2}_{T}\right)\\
    & \ +Crt^{\frac{1}{2}}\left(\vertiii{S_1-S_2}_{T}+\vertiii{I_1-I_2}_{T}+\vertiii{R_1-R_2}_{T}\right).
\end{align*}
Using very similar computations for the other two components $\mathcal{T}_2$ and $\mathcal{T}_3$ (we omit the details), we obtain
\begin{align}
    & \vertiii{\mathcal{T}(S_1,I_1,R_1)-\mathcal{T}(S_2,I_2,R_2)}_T\nonumber\\
& \ \leq C\left(\vertiii{S_{1}-S_{2}}_{T}+\vertiii{I_{1}-I_{2}}_{T}+\vertiii{R_{1}-R_{2}}_{T}\right)\left(r T+ T^{\frac{1}{2}}+ T\right),\label{eq;estimate2}
\end{align}
for a suitable positive constant $C$.
Combining the two estimates \eqref{eq:estimate1} and \eqref{eq;estimate2}, in order to have the map $\mathcal{T}$ well-defined and a strict contraction we need the two constants $r>0$ and $T>0$ such that the following inequalities are satisfied:
\begin{align}
   & C\left(1+ (r+r^2)(T^{1/2}+T)\right)\leq r\label{eq:estimate_wp},\\
   &C\left( T+ r T^{\frac{1}{2}} + T\right) < 1\label{eq:estimate_con}.
\end{align}
It is easily seen that a suitable choice of $r$ large enough and $T$ small enough allows to fulfill both inequalities \eqref{eq:estimate_wp} and \eqref{eq:estimate_con}.
Hence, Banach's fix point implies the existence of a unique $\left(S,I,R\right)\in X_{r,T}$ such that $\mathcal{T} \left(S,I,R\right)=\left(S,I,R\right)$. The solution to the fixed point problem is actually a classical solution to the Cauchy problem for \eqref{mainsystem}. This is a consequence of Duhamel's principle. We omit the details.
\end{proof}

\begin{Remark}\label{Rem:localExt}
    \emph{
    The local existence result in Proposition \ref{prop:local} can be easily extended to any Cauchy problem of the form
    \begin{equation}
        \begin{dcases}
            \partial_t u_i(x,t)=\varepsilon\Delta u_i + \mathrm{div}(u_i \nabla V_i[u_1,\ldots,u_N])+g_i(u_1,\ldots,u_N), & x\in \R^d\,,\,\, t>0\\
            u_i(x,0)=u_i^0(x), &
        \end{dcases}
    \end{equation}
    where $i=1,\ldots,N$, with nonlocal operators of the form
    \[V_i[u_1,\ldots,u_N](x)=\sum_{j=1}^N W_{ij}\ast u_i(x)\,,\]
    provided the interaction kernels $W_{ij}$ satisfy the same assumptions as in \eqref{eq:ass_reg_W}, the functions $g_i:\R^N\rightarrow \R$ feature $C^1$ regularity, and $u_i^0\in L^1\cap L^\infty$. The space dimension $d$ is arbitrary. A short version of the proof is provided in the appendix.
    }
\end{Remark}

\subsection{Non-negativity of local solutions}\label{subsec:positivity}
The solution provided in Subsection \ref{subsec:local} may, in principle, violate the conditions $S(x,t)\geq0$, $I(x,t)\geq 0$, $R(x,t)\geq 0$. This would make them inconsistent with applications to epidemiology. Hence, assuming the initial conditions satisfy
\begin{equation}\label{eq:initial_nonneg}
    S_0(x)\geq 0\,,\qquad I_0(x)\geq 0\,,\qquad R_0(x)\geq 0\,,
\end{equation}
we shall prove that the three functions $S, I, R$ are nonnegative throughout their whole existence lifespan. For future use, we consider a family of functions $\eta_{\delta}:\mathbb{R} \rightarrow [0,+\infty)$ parametrised by $\delta>0$ satisfying
\begin{equation}\label{etaC2}
    \eta_{\delta} \in C^{2}\left(\mathbb{R}\right),
\end{equation}
\begin{equation}\label{etaprime2pos}
    \eta^{\prime\prime} _{\delta}\geq 0,
\end{equation}
\begin{equation}\label{eq:etaprime}
    \eta_\delta'(x)=0\qquad \hbox{for all $x\geq 0$},
\end{equation}
\begin{equation}\label{etaCon}
    \eta_{\delta}(x) \rightarrow (x)_{-} \quad\text{uniformly on $\R$ as}\quad \delta \rightarrow 0,
\end{equation}
\begin{equation}\label{etaprimeIneq}
    \left|\eta^{\prime}_{\delta}(x) x\right| \leq (x)_-.
\end{equation}
We also introduce the sequence of cut-off functions $\zeta_{n}: \mathbb{R}^{2} \rightarrow[0,1]$ such that, for any $n \in \mathbb{N}$, $\zeta_{n} \in C^{2}(\mathbb{R})$ and
\begin{align}
 & \zeta_{n}(x)= \begin{cases}1, & x \in B_{n}(0) \label{eq:cutoff}\\
0, & x \notin B_{n+1}(0)\end{cases} \\
& \left|\nabla \zeta_{n}\right|+\left|\Delta \zeta_{n}\right| \leq C \left|\zeta_n\right|.\label{eq:cutoff_est}
\end{align}
for some constant $C$ independent of $n$.

\begin{Lemma}[Non-negativity of solutions]\label{lem:nonneg}
    Under the same assumptions of Proposition \ref{prop:local}, assuming further \eqref{eq:initial_nonneg}, the local solution $(S,I,R)$ provided in Proposition \ref{prop:local} is nonnegative on its maximal existence time interval $[0,T]$.
\end{Lemma}
\begin{proof}
Let $(S,I,R)$ be the solution provided in Proposition \ref{prop:local} on the maximal existence time interval $[0,T]$. We omit the $(x,t)$-dependency for simplicity and compute
\begin{align}
    &  \frac{d}{d t} \int_{\mathbb{R}^{2}} \eta_{\delta}(S) \zeta_{n} d x=\int_{\mathbb{R}^{2}} \eta_{\delta}^{\prime}(S) S_{t}\zeta_{n} d x=\int_{\mathbb{R}^{2}} \eta_{\delta}^{\prime}(S)\left(\operatorname{div}\left(S \nabla V_{S}\right)+\varepsilon \Delta S-\beta S I\right) \zeta_{n} d x\nonumber\\
 & \ = \int _{\mathbb{R}} \eta_{\delta} ^{\prime} \left(S\right)\nabla S\cdot\nabla V_{S}[S,I,R] \zeta_{n}dx+ \int_{\mathbb{R}^2}\eta_{\delta}^{\prime}\left(S\right)S\Delta V_S[S,I,R]\zeta_n dx - \beta\int_{\mathbb{R}^2} \eta_{\delta}^{\prime}\left(S\right)SI\zeta_{n}dx\nonumber\\
 & \ -\varepsilon\int_{\mathbb{R}^2}\eta_{\delta}^{\prime\prime}\left(S\right)\left|\nabla S\right|^2 \zeta_n dx-\varepsilon\int_{\mathbb{R}^2}\eta_{\delta}^{\prime} \left(S\right)\nabla S \cdot \nabla\zeta_n dx\,,\label{eq:est_pos_1}
\end{align}
where we have used integration by parts on the terms involving the linear diffusion term using that the functions under the integral sign are compactly supported. Using \eqref{etaprime2pos} and the fact that $\zeta_n$ is a non-negative function we have
$$
\begin{gathered}
    -\int_{\mathbb{R}^2}\varepsilon\eta _{\delta}^{\prime\prime} (S)\left|\nabla S\right|^2 \zeta_n(x)\leq 0.\\
\end{gathered}
$$
Moreover,
$$
 -\int _{\mathbb{R}^2} \varepsilon \eta_{\delta}^{\prime}(S) \nabla S \cdot \nabla \zeta_n dx =\int _{\mathbb{R}^2} \varepsilon \eta_{\delta}(S) \Delta\zeta_n dx.
$$
We also have
\begin{align*}
    & \int_{\mathbb{R}^{2}} \eta_{\delta}^{\prime}\left(S\right) \nabla S \cdot \nabla V_{S}[S, I, R] \zeta_{n} d x=\int_{\mathbb{R}^{2}} \nabla\left(\eta_{\delta}(S)\right)\cdot \nabla V_{S}[S, I, R] \zeta_{n} d x \\
 & \ =-\int_{\mathbb{R}^{2}} \eta_{\delta}(S) \Delta V_{S}[S, I, R] \zeta_{n} d x-\int_{\mathbb{R}^{2}} \eta_{\delta}(S) \nabla V_{S}[S, I, R]\cdot \nabla \zeta_{n} d x .
\end{align*}
Combining the above computations with \eqref{eq:est_pos_1}, recalling that $S,I$ and $R$ belong to $L^1\cap L^\infty$ on $t\in [0,T]$, we obtain the estimate
\begin{align*}
    & \frac{d}{d t} \int_{\mathbb{R}^{2}} \eta_{\delta}(S) \zeta_{n} d x
    \leq
 \left\|\Delta V_{S}[S, I, R]\right\|_{L^{\infty}} \int_{\mathbb{R}^{2}} \eta_{\delta}(S)  d x + C\left\|\nabla V_{S}[S, I, R]\right\|_{L^{\infty}} \int_{\mathbb{R}^{2}} \eta_{\delta}(S)  d x \\
 & \ +C\varepsilon \int_{\mathbb{R}^{2}} \eta_{\delta}(S)   d x+\beta\|I\|_{L^{\infty}} \int_{\mathbb{R}^{2}}\left|\eta_{\delta}^{\prime}(S) S\right| d x\\
 & \ +\left\|\Delta V[S, I, R]\right\|_{L^{\infty}} \int_{\mathbb{R}^{2}} \left|\eta_{\delta}^{\prime}(S) S \right|  d x  \leq C \int_{\mathbb{R}^{2}} \eta_{\delta}(S) d x+C\int_{\mathbb{R}^{2}} (S)_{-} dx,
\end{align*}
for a suitable constant $C>0$, where we have used once again \eqref{eq:est_kernels} with $p=+\infty$.
We now let $\delta \rightarrow 0$. To do that, we first integrate the above estimate in time to get
\begin{align*}
    & \int_{\mathbb{R}^{2}} \eta_{\delta}(S(x,t)) \zeta_{n}(x) d x \leq \int_{\mathbb{R}^{2}} \eta_{\delta}(S_0(x)) \zeta_{n}(x) d x \\
    & \ + C\int_0^t  \int_{\mathbb{R}^{2}} \eta_{\delta}(S(x,s) d x ds + C \int_0^t \int_{\mathbb{R}^{2}} (S(x,s))_{-} dx ds\,.
\end{align*}
We now use dominated convergence theorem to get
\begin{align*}
    & \int_{\mathbb{R}^{2}} (S(x,t))_- \zeta_{n}(x) d x \leq \int_{\mathbb{R}^{2}} (S_0(x))_- \zeta_{n}(x) d x  + C\int_0^t  \int_{\mathbb{R}^{2}} (S(x,s))_- d x ds\,.
\end{align*}
for a suitable $C\geq 0$. We now let $n \rightarrow \infty$ and obtain
$$
\int_{\mathbb{R}^{2}}(S(x, t))_{-} d x \leq \int_{\mathbb{R}^{2}}(S(x,0))_{-} d x +C \int_0 ^t \int_{\mathbb R^2} (S(x,s))_{-}dx ds.
$$
Grönwall's lemma implies
$$
\int_{\mathbb{R}^{2}}(S(x, t))_{-} d x \leq  e^{C t} \int_{\mathbb{R}^{2}}\left(S_{0}\right)_{-} d x = 0,
$$
since $S_0$ is non-negative. The above implies $S(x,t)$ is nonnegative almost everywhere (and, by regularity, everywhere) on $\R^2\times [0,T]$.
We now apply the same method for $I(x,t)$. We estimate
\begin{align*}
    & \frac{d}{d t} \int_{\mathbb{R}^{2}} \eta_{\delta}(I) \zeta_{n} d x =
\int_{\mathbb{R}^{2}} \eta_{\delta}^{\prime}(I)\left(\operatorname{div}\left(I \nabla V_{I}[S,I,R]\right)+\varepsilon \Delta I+\beta S I-\alpha I\right) \zeta_{n} d x \\
& \ =  \int_{\mathbb{R}^{2}} \eta_{\delta}^{\prime}(I) \nabla I \cdot \nabla V_{I}[S, I, R] \zeta_{n} d x+\alpha\int_{\mathbb{R}^{2}}  \eta_{\delta}^{\prime}(I) I \zeta_{n} d x\\
& \ +\int_{\mathbb{R}^{2}} \eta_{\delta}^{\prime}(I) I \Delta V_{I}[S, I, R] \zeta_{n} d x + \beta\int_{\mathbb{R}^{2}} \eta_{\delta}^{\prime}(I) S I \zeta_{n} d x \\
 & \ -\varepsilon\int_{\mathbb{R}^{2}}  \eta_{\delta}^{\prime \prime}(I)|\nabla I|^{2} \zeta_{n} d x-\varepsilon \int_{\mathbb{R}^{2}} \eta_{\delta}^{\prime}(I) \nabla I \cdot \nabla  \zeta_{n}d x \,.
\end{align*}
Similarly to the case of $S(x,t)$, we obtain
\begin{align*}
    & \frac{d}{d t} \int_{\mathbb{R}^{2}} \eta_{\delta}(I) \zeta_{n}(x) d x\\
 & \ \leq\left\|\Delta V_{I}[S, I, R]\right\|_{L^{\infty}} \int_{\mathbb{R}^{2}} \eta_{\delta}(I)  d x+ C \left\|\nabla V_{I}[S, I, R]\right\|_{L^{\infty}} \int_{\mathbb{R}^{2}} \eta_{\delta}(I)  d x\\
 & \ +\alpha \int_{\mathbb{R}^{2}}\left|\eta_{\delta}^{\prime}(I) I\right|  d x  +\varepsilon C \int_{\mathbb{R}^{2}} \eta_{\delta}(I) d x\\
 & \ +\beta\|S\|_{L^{\infty}} \int_{\mathbb{R}^{2}}\left|\eta_{\delta}^{\prime}(I) I\right| d x+\|\Delta V_{I}[S, I, R]\|_{L^{\infty}} \int_{\mathbb{R}^{2}}\left|\eta_{\delta}^{\prime}(I) I\right|  d x \\
 & \ \leq C\int_{\mathbb{R}^2} \eta_{\delta}(I) dx +C\int_{\mathbb{R}^2} \left(I\right)_{-} dx.
\end{align*}
for a suitable constant $C$. From the above estimate we may easily conclude the nonnegativity of $I$ as in the estimate for $S$. The same method also applies for $R$, where we may benefit of the fact that $I\geq 0$, therefore, the only reaction term contributes to maintain the non-negative sign of the initial condition. The details are left to the reader.
\end{proof}

We prove at this stage an important property of the solution on the maximal existence time interval $[0,T]$.

\begin{Lemma}[Conservation of the total mass]\label{lem:conservation}
    Under the assumptions of Proposition \ref{prop:local}, given
    \[M(t):=\int_{\R^2}\left(S(x,t)+I(x,t)+R(x,t)\right) dx\,,\]
    the function $[0,T]\ni t\mapsto M(t)$ is constant.
\end{Lemma}

\begin{proof}
    Given the cutoff $\zeta_n$ defined in \eqref{eq:cutoff}, we compute
    \begin{align*}
        & \frac{d}{dt}\int_{\R^2}\left(S+I+R\right) \zeta_ndx\\
        & \ =\varepsilon\int_{\R^2}\Delta(S+I+R)\zeta_n dx+\int_{\R^2}\mathrm{div}(S\nabla V_S[S,I,R] + I\nabla V_I([S,I,R]) + R\nabla V_R[S,I,R])\zeta_n dx \\
        & \ = \varepsilon \int_{\R^2}(S+I+R)\Delta \zeta_n dx -\int_{\R^2}\left(S\nabla V_S[S,I,R] + I\nabla V_I([S,I,R]) + R\nabla V_R[S,I,R]\right)\cdot \nabla \zeta_n dx\,.
    \end{align*}
    We integrate the above identity in time on $\tau\in [0,t]$ for $t\in [0,T)$ and obtain
    \begin{align*}
        & \int_{\R^2}\left(S+I+R\right) \zeta_ndx=\int_{\R^2}\left(S_0+I_0+R_0\right) \zeta_ndx + \int_0^t \left[\varepsilon \int_{\R^2}(S+I+R)\Delta \zeta_n dx \right.\\
        & \ \left.-\int_{\R^2}\left(S\nabla V_S[S,I,R] + I\nabla V_I([S,I,R]) + R\nabla V_R[S,I,R]\right)\cdot \nabla \zeta_n dx\right]d\tau\,.
    \end{align*}
    Now, since the functions $(S+I+R)$ and $\left(S\nabla V_S[S,I,R] + I\nabla V_I([S,I,R]) + R\nabla V_R[S,I,R]\right)$ are uniformly bounded with respect to $n$ in $L^1\cap L^\infty$ on $\R^2\times [0,t]$ due to Proposition \ref{prop:local} and to Lemma \ref{lem:nonlocal_etsimates} applied to the cases $p=1$ and $p=+\infty$, we may send $n\rightarrow+\infty$ and use that $\Delta \zeta_n$ and $\nabla \zeta_n$ converge to zero almost everywhere on $\R^2$ to obtain the desired assertion via Lebesgue dominated convergence.
\end{proof}

\subsection{Global existence for the approximated problem}\label{subsec:global}
In this section, we show that the local, nonnegative solution found in the previous subsections is, in fact, a global solution. To perform this task, we will show that the $L^p$ norms of the solution $(S,I,R)$ do not blow up on a bounded maximal existence time interval $[0,T]$ for all $p\in [1,+\infty]$. This will be achieved via direct $L^p$ estimates on the solution, which are justified by the fact that our solution is classical and we may therefore suitably differentiate it and we may integrate by parts in the estimates below. Global existence then follows from a standard continuation principle. We emphasise that the estimates we obtain here are uniform with respect to $\varepsilon>0$. In particular, all generic constants $C$ in this subsection are independent of $\varepsilon$.

\begin{Proposition}[Uniform $L^p$ estimates and global existence]\label{propLp}
Let $(S,I,R)$ be the unique local solution to the system \eqref{mainsystem} found in Proposition \ref{prop:local} and let $p\in [1,+\infty]$. Under the same assumptions of Proposition \ref{prop:local}, assuming further $S_{0}, I_{0}, R_{0} \in\left(L^{1} \cap L^{\infty}\left(\mathbb{R}^{2}\right)\right)$ are nonnegative functions, then, there exists a constant $C\geq 0$ independent of time and of $\varepsilon$ such that
\begin{align}
   &  \|S(\cdot,t)\|_{L^p(\R^2)}+ \|I(\cdot,t)\|_{L^p(\R^2)}+ \|R(\cdot,t)\|_{L^p(\R^2)}\nonumber\\
   & \ \leq \left(1+\|S_0\|_{L^p(\R^2)}+\|I_0\|_{L^p(\R^2)}+\|R_0\|_{L^p(\R^2)}\right) e^{Ct(1+e^{Ct})},\label{eq:Lp}
\end{align}
for all $t\in [0,T]$. Consequently, the solution $(S,I,R)$ found in Proposition \ref{prop:local} exists globally in time.
\end{Proposition}

\begin{proof}
\textbf{Step 1 - $L^p$ estimate of $S$ for $p\in [1,+\infty]$.}
Consider the cut-off function $\zeta_n$ defined in \eqref{eq:cutoff}. By $L^p$-interpolation inequality, the initial data $S_0,I_0,R_0$ are in $L^p (\mathbb R^2)$ for any $p\in [1,\infty]$. For finite $p$, we integrate by parts and compute
\begin{align*}
    &  \frac{d}{d t} \int_{\mathbb{R}^{2}} S^{p} \zeta_{n} d x=p\int_{\mathbb{R}^{2}} S^{p-1} S_{t} \zeta_{n}d x=p \int_{\mathbb{R}^{2}} S^{p-1}\left(\varepsilon \Delta S+\operatorname{div}\left(S \nabla V_{S}[S, I, R]\right)-\beta S I\right) \zeta_{n}d x \\
& \ =-p(p-1) \varepsilon \int_{\mathbb{R}^{2}} S^{p-2}|\nabla S|^{2} \zeta_{n} d x-p \varepsilon \int_{\mathbb{R}^{2}}S^{p-1} \nabla S \cdot \nabla \zeta_{n} d x\\
& \ -p(p-1) \int_{\mathbb{R}^{2}} S^{p-1} \nabla S \cdot\nabla V_{S}[S, I, R] \zeta_{n} d x -p \int_{\mathbb{R}^{2}} S^{p} \nabla V_{S}[S, I, R]\cdot \nabla \zeta_{n}d x-\beta p\int_{\mathbb{R}^{2}} S^{p} I \zeta_{n} d x\,.
\end{align*}
Since the functions $S,I$ and $\zeta_n$ are nonnegative, we have
\begin{align*}
    & -p(p-1) \varepsilon \int_{\mathbb{R}^{2}} S^{p-2}|\nabla S|^{2} \zeta_{n}(x) d x\leq 0,\\
& \ -\beta p\int_{\mathbb{R}^{2}} S^{p} I \zeta_{n}(x) d x\leq 0.
\end{align*}
Also, $p S^{p-1} \nabla S = \nabla(S^p)$, so using the integration by parts we have
\begin{align*}
    & -p(p-1) \int_{\mathbb{R}^{2}} S^{p-1} \nabla S \cdot\nabla V_{S}[S, I, R] \zeta_{n}(x) d x\\
& \ = (p-1) \int_{\mathbb{R}^2}S^p \Delta V_S[S,I,R]\zeta_n(x) dx + (p-1)\int_{\mathbb{R}^2} S^p \nabla V_S[S,I,R]\cdot \nabla \zeta_n(x)dx
\end{align*}
and
\begin{align*}
    & -p \varepsilon \int_{\mathbb{R}^{2}}S^{p-1} \nabla S \cdot \nabla \zeta_{n}(x) d x =\varepsilon\int_{\mathbb{R}^2} S^p \Delta \zeta_n(x)dx.
\end{align*}
We then use \eqref{eq:cutoff_est} and obtain, for a suitable constant $C\geq 0$, the estimate
\begin{align*}
    & \frac{d}{d t} \int_{\mathbb{R}^{2}} S^{p} \zeta_{n}(x) d x \leq  \varepsilon \int_{\mathbb{R}^{2}} S^{p} \Delta \zeta_{n}(x) d x\\
 & \ +(p-1) \int_{\mathbb{R}^{2}} S^{p} \Delta V_{S}[S, I, R] \zeta_{n}(x) d x-\int_{\mathbb{R}^{2}} S^{p} \nabla V_{S}[S, I, R] \nabla \zeta_{n}(x) d x \\
 & \ \leq \varepsilon C \|S\|_{L^{p}}^{p}+(p-1)\left\|\Delta V_{S}[S, I, R]\right\|_{L^{\infty}}\left\| S\right\|_{L^{p}}^{p}\\
 & \ +C\left\| \nabla V_{S}[S, I, R]\right\|_{L^{\infty}}\left\| S \right\|_{L^{p}}^{p} \\
 & \ \leq \left\| S \right\|_{L^{p}}^{p}\left(\varepsilon C + (p-1)\left\|\Delta V_{S}[S, I, R]\right\|_{L^{\infty}}+C\left\| \nabla V_{S}[S, I, R]\right\|_{L^{\infty}}\right)\,.
\end{align*}
Using \eqref{eq:est_kernels} with $p=1$, we may find a constant $C\geq 0$ independent of $T$ and of $\varepsilon$ such that
 $$
 \left \|\nabla V_S[S,I,R]\right\|_{L^{\infty}(\mathbb{R}^2)}+\left \|\Delta V_S[S,I,R]\right\|_{L^{\infty}(\mathbb{R}^2)}\leq C\left(\left\|S\right\|_{L^1(\mathbb{R}^2)}+\left\|I\right\|_{L^1(\mathbb{R}^2)}+\left\|R\right\|_{L^1(\mathbb{R}^2)}\right)\,.
 $$
Now, since $S$, $I$, and $R$ are nonnegative on their existence lifespan due to Lemma \ref{lem:nonneg}, we have
\[\left(\left\|S\right\|_{L^1(\mathbb{R}^2)}+\left\|I\right\|_{L^1(\mathbb{R}^2)}+\left\|R\right\|_{L^1(\mathbb{R}^2)}\right) = \int_{\R^2}(S(x,t)+I(x.t)+R(x,t))dx\,.\]
Lemma \ref{lem:conservation} implies the above quantity is constant in time and equal to
\[M_0=\int_{\R^2}(S_0+I_0+R_0)dx\,.\]
Hence, we may easily find a constant $C\geq 0$ depending on the interaction kernels and on $M_0$ such that
 $$
\frac{d}{d t} \int_{\mathbb{R}^{2}} S^{p} \zeta_{n}(x) d x \leq C \int_{\mathbb{R}^{2}} S^{p} d x\,.
$$
Integrating in time, letting $n\rightarrow +\infty$ and via Gronwall's inequality we obtain
\begin{equation}\label{eq:S_Lp}
    \|S\|_{L^{p}\left(\mathbb{R}^{2}\right)} \le e^{Ct}\left\|S_{0}\right\|_{L^{p}\left(\mathbb{R}^{2}\right)},
\end{equation}
for all $t\in [0,T]$, where $C$ is a constant and depends on the interaction kernels, the initial total mass $M_0$ and constant coefficients $\alpha,\beta$. In order to obtain the estimate of $S$ in $L^\infty$, we use \eqref{eq:S_Lp} as follows:
\begin{align}
    &  \|S(\cdot,t)\|_{L^{\infty}}\leq \limsup_{p \nearrow+\infty}\|S(\cdot,)\|_{L^{p}} \leq \limsup_{p \nearrow +\infty} e^{C t} \| S_{0} \|_{L^{p}}\leq \|S_0\|_{L^\infty}e^{Ct}\,,\label{eq:S_L_inf}
\end{align}
where the first inequality above is standard and the
last one is obtained via standard $L^p$ interpolation.

\noindent\textbf{Step 2 - $L^p$ estimate of $I$ for $p\in [1,+\infty]$.}
Let us now estimate the $L^p$ norm of $I$ for finite $p$. Similarly as for $S$, we compute
\begin{align*}
    & \frac{d}{d t} \int_{\mathbb{R}^{2}} I^{p}(x, t) \zeta_{n}(x) d x=p \int_{\mathbb{R}^{2}} I^{p-1} I_{t} \zeta_{n}(x) d x \\
& \ =p \int_{\mathbb{R}^{2}} I^{p-1}\left(\varepsilon \Delta I+\operatorname{div}\left(I \nabla V_{I}[S, I, R]\right)+\beta S I-\alpha I\right) \zeta_{n}(x) d x \\
& \ =-\varepsilon p(p-1) \int_{\mathbb{R}^{2}} I^{p-2}|\nabla I|^{2} \zeta_{n}(x) d x-\varepsilon p \int_{\mathbb{R}^{2}} I^{p-1} \nabla I\cdot \nabla \zeta_{n}(x) d x\\
& \ -p(p-1) \int_{\mathbb{R}^{2}} I^{p-1} \nabla I \cdot\nabla V_{I}[S, I, R] \zeta_{n}(x) d x -\alpha p\int_{\mathbb{R}^{2}} I^{p} \zeta_{n}(x) d x \\
& \ -p \int_{\mathbb{R}^{2}} I^{p} \nabla V_{I} \cdot\nabla \zeta_{n}(x) d x+\beta p\int_{\mathbb{R}^{2}} I^{p} S \zeta_{n}(x) d x\\
& \ \leq \varepsilon \int_{\mathbb{R}^{2}} I^{p} \Delta \zeta_{n}(x) d x+(p-1) \int_{\mathbb{R}^{2}} I^{p} \Delta V_{I}[S, I, R] \zeta_{n}(x) d x\\
& \ +(p-1) \int_{\mathbb{R}^{2}} I^{p} \nabla V_{I}[S, I, R] \cdot \nabla \zeta_{n}(x) d x \\
& \ -p \int_{\mathbb{R}^{2}} I^{p} \nabla V_{I}[S, I, R] \cdot \nabla \zeta_{n}(x) d x+\beta p\int_{\mathbb{R}^2} S I^{p} \zeta_{n}(x) d x \\
& \ \leq \varepsilon C \int_{\mathbb{R}^{2}} I^{p} \zeta_n(x) d x+(p-1) C \|\Delta V_I[S,I,R]\|_{L^{\infty}} \int_{\mathbb{R}^{2}} I^{p} \zeta_n(x) d x\\
& \ +\left\|\nabla V_{I}[S, I, R]\right\|_{L^{\infty}} \int_{\mathbb{R}^{2}} I^{p}\zeta_n (x) dx+\beta p\|S_0\|_{L^{\infty}}e^{Ct} \int_{\mathbb{R}^{2}} I^{p}\zeta_n(x) d x,
\end{align*}
where the last estimate is justified by \eqref{eq:S_L_inf}. Hence, similarly to \eqref{eq:S_Lp} we use Lemma \ref{lem:nonlocal_etsimates} and obtain, for a suitable constant $C$ independent of $p$, of $\varepsilon$ and of time,
\begin{equation*}
    \frac{d}{d t} \int_{\mathbb{R}^{2}} I^{p}(x, t) \zeta_{n}(x) d x\leq C(1+p)(1+e^{Ct}) \int_{\R^2}I^p(x,t) dx\,.
\end{equation*}
Hence, by integrating in time we
conclude via monotone convergence and Gronwall's lemma
\begin{equation}\label{eq:I_Lp}
    \int_{\mathbb{R}^{2}} I^{p}(x,t)  d x \leq e^{C(1+p)t(1+e^{Ct})} \int_{\mathbb{R}^{2}} I_{0}^{p}(x) d x,
\end{equation}
where the constant $C$ is independent of $t$, of $p$, and of $\varepsilon$ (for small $\varepsilon>0$). The $L^\infty$ estimate for $I$ follows similarly to what is done for $S$.

\noindent\textbf{Step 3 - $L^p$ estimate of $R$ for $p\in [1,+\infty]$.}
We now perform a similar estimate of the $L^p$ of $R$ for finite $p$. We compute
\begin{align*}
    & \frac{d}{d t} \int_{\mathbb{R}^{2}} R^{p}(x, t) \zeta_{n}(x) d x=p \int_{\mathbb{R}^{2}} R^{p-1} R_{t} \zeta_{n}(x) d x\\
& \ =p \int_{\mathbb{R}^{2}} R^{p-1}\left(\varepsilon \Delta R+\operatorname{div}\left(R \nabla V_{R}[S, I, R]\right)+\alpha I\right) \zeta_{n}(x) d x \\
& \ =-p(p-1) \varepsilon \int R^{p-2}|\nabla R|^{2} \zeta_{n} d x-p \varepsilon \int_{\mathbb{R}^{2}} R^{p-1} \nabla R \cdot \nabla \zeta_{n} d x \\
 & \ -p(p-1) \int_{\mathbb{R}^{2}} R^{p-1} \nabla R\cdot \nabla V_{R}[S,I,R] \zeta_{n} d x\\
 & \ -p \int_{\mathbb{R}^{2}} R^{p} \nabla V_{R}[S,I,R]\cdot \nabla \zeta_{n} d x +\alpha p \int_{\mathbb{R}^{2}} I R^{p-1} \zeta_{n} d x \\
& \  \leq \varepsilon \int_{\mathbb{R}^{2}} R^{p} \Delta \zeta_{n}(x) d x +(p-1) \int_{\mathbb{R}^{2}} R^{p} \Delta V_{R} \zeta_{n} d x\\
 & \ +(p-1) \int_{\mathbb{R}^{2}} R^{p} \nabla V_{R} \cdot \nabla \zeta_{n} d x -p \int_{\mathbb{R}^{2}} R^{p} \nabla V_{R} \cdot \nabla \zeta_{n} d x+\alpha p \int_{\mathbb{R}^{2}} I R^{p-1} \zeta_{n} d x \\
 & \ \leq \varepsilon C \int_{\mathbb{R}^{2}} R^{p} d x+(p-1)\left\|\Delta V_{R}\right\|_{L^{\infty}} \int_{\mathbb{R}^{2}} R^{p}d x\\
 & \ +\left\|\nabla V_{R}\right\|_{L^{\infty}} \int_{\mathbb{R}^{2}} R^{p} d x+\alpha p\|I\|_{L^{p}}\| R\|_{L^{p}}^{p-1}\,,
\end{align*}
where we have used Hoelder's inequality to control the last term. We now use Young's inequality as follows,
\begin{align*}
    &  \|I\|_{L^{p}} \|R\|_{L^{p}}^{p-1} \leq \frac{1}{p}\|I\|_{L^{p}}^{p}+\frac{p-1}{p}\| R\|_{L^{p}}^{p}\\
 & \ \leq \frac{1}{p}\left(e^{(p+1)Ct(1+e^{Ct})}\left\|I_{0}\right\|_{L^{p}}^{p}\right)+\frac{p-1}{p}\| R\|_{L^{p}}^{p},
 \end{align*}
 which implies, once again due to Lemma \ref{lem:nonlocal_etsimates},
 \begin{align*}
&\frac{d}{d t} \int_{\mathbb{R}^{2}} R^{p}(x, t) \zeta_{n}(x) d x \leq p C \int_{\mathbb R^{2}} R^{p}  d x+C_1(t)
\end{align*}
for a suitable constant $C$ independent of $\varepsilon$, of $p$ and of time, and with
\[C_1(t)=Ce^{(p+1)Ct(1+e^{Ct})}\,.\]
By integrating in time, by letting
$n \rightarrow \infty$, by using the Monotone Convergence theorem and Gronwall's inequality, we get
\begin{align*}
    & \int_{\mathbb{R}^{2}} R^{p}(x, t)  d x \leq e^{p Ct}\left(\int_{\mathbb{R}^{2}} R_{0}^{p} d x+C\int_0 ^t e^{(p+1)Cs(1+e^{Cs})} d s\right)\\
    & \ \leq e^{pCt}\int_{\mathbb{R}^{2}} R_{0}^{p} d x+Ct e^{(p+1)Ct(2+e^{Ct})}\,,
\end{align*}
which implies, for finite $p$,
\begin{equation}\label{eq:R_Lp}
    \|R(\cdot,t)\|_{L^p(\R^2)}\leq e^{Ct}\left(\|R_0\|_{L^p(\R^2)}+Cte^{2Ct(2+e^{Ct})}\right),
\end{equation}
and the estimate for $p=+\infty$ follows as in the previous steps. Combining \eqref{eq:S_Lp}, \eqref{eq:I_Lp}, and \eqref{eq:R_Lp}, we obtain \eqref{eq:Lp}.

\noindent\textbf{Step 4. Global existence.} Assuming by contradiction that the maximal existence time $T$ in Proposition \ref{prop:local} is finite, then estimate \eqref{eq:Lp} implies we could extend the solution by continuity to time $t=T$ and prove existence and uniqueness of solutions to the Cauchy problem starting at time $t=T$, with local existence for
some further interval $[T, T+\tau)$ for some positive $\tau>0$. This contradicts $T$ being maximal.
\end{proof}

\begin{Remark}
    \emph{The estimate in \eqref{eq:Lp} is far from being sharp. The only goal in Proposition \ref{propLp} is to find a globally defined function of time controlling the $L^p$ norms of $S$, $I$, and $R$.}
\end{Remark}

\section{Existence via vanishing diffusion}\label{sec:existence}

Having proven that for given nonnegative initial conditions $S_0$, $I_0$, and $R_0$, there exists a unique classical solution to the approximated problem \eqref{mainsystem}, which exists globally in time, we are now ready to tackle the $\varepsilon\searrow 0$ limit problem. The result in proposition \ref{propLp} implies that for all $T\geq 0$ the family of solutions $(S_\varepsilon,I_\varepsilon,R_\varepsilon)$ to \eqref{mainsystem} parametrised by $\varepsilon$ is weakly compact in $L^p$ for finite $p>1$ (resp. weakly-$*$ compact in $L^\infty$) on the space-time domain $\R^2\times [0,T]$ provided the initial functions $S_0, I_0, R_0$ belong to $L^p(\R^d)$. In the context of nonlocal interaction equations (or systems of equations) of the form \eqref{eq:nonlocal_intro}, this would be enough to guarantee the existence of a subsequence that converges weakly to the solution to the $\varepsilon=0$ limiting equation \eqref{eq:nonlocal_intro}. In our case, however, the presence of the nonlinear reaction terms $\beta SI$ and $-\beta SI$ in system \eqref{mainsystem} requires a stronger form of compactness.

\subsection{Uniform $H^1$ estimates}\label{subsec:H1}

In this subsection we assume the initial data belong to $H^1(\R^2)$, and we prove that under such condition, the family $(S_\varepsilon,I_\varepsilon,R_\varepsilon)$ is strongly compact in $L^2$. In order to perform said estimates, we will implicitly assume that we may differentiate the PDEs in \eqref{mainsystem} with respect to $x$, which implies a higher regularity for the solution $(S,I,R)$ compared to $C^{1,2}_{t,x}$. However, the fixed point argument performed in subsection \ref{subsec:local} may be easily adapted to consider an initial condition in $H^1$. More precisely, the functional space $X_{r,T}$ in Proposition \ref{prop:local} should include the $L^2$ norm of the gradient as well. This will imply that we may differentiate the operator $\mathcal{T}(\tilde{S},\tilde{I},\tilde{R})$ with respect to $x$ and perform all estimates needed to prove the well-posedness of the operator and contraction. This procedure is, no matter how tedious and lengthy, pretty standard. We shall, therefore omit the details.

\begin{Proposition}\label{prop:H1}
Let $(S,I,R)$ be the unique solution to the system \eqref{mainsystem} found in section \ref{sec:approximated}. Assume the interaction kernels satisfy \eqref{eq:ass_reg_W} and that the initial functions $S_{0}, I_{0}, R_{0}$ belong to $L^{1}(\R^2) \cap L^{\infty}(\R^2)\cap H^1(\R^2)$ and are nonnegative. Then, for a given $T\geq 0$, the $H^1$-norm of $(S(\cdot,t),I(\cdot,t),R(\cdot,t))$ is uniformly bounded on $[0,T]$ with respect to $\varepsilon>0$.
\end{Proposition}

\begin{proof}
Let $\zeta_n$ be the cut-off function defined in \eqref{eq:cutoff}. We first compute
\begin{align}
    & \frac{d}{dt} \int_{\mathbb R^2} (S_{x_1}) ^2 \zeta_n dx = 2 \int _{\mathbb R ^2} S_{x_1} \left(S_{x_1}\right)_t \zeta_n dx\nonumber\\
& \ =2\varepsilon \int_{\mathbb R^2} S_{x_1}\left(\Delta S\right)_{x_1} \zeta_n dx +2\int_{\mathbb R^2} S_{x_1}\left(\operatorname{div}\left(S\nabla V_S[S,I,R]\right)\right)_{x_1}\zeta_n dx-2\beta \int_{\mathbb R^2} S_{x_1}(SI)_{x_1} \zeta_n dx.\label{eq:H1_1}
\end{align}

We consider the three terms on the right-hand side of \eqref{eq:H1_1} and estimate each of them separately. For the first term we have
\begin{align*}
    & 2\varepsilon \int_{\mathbb R^2} S_{x_1} \left(\Delta S\right)_{x_1} \zeta_n dx =-2\varepsilon \int_{\mathbb R^2} S_{x_1x_1} \Delta S \zeta_n dx -2\varepsilon \int_{\mathbb R^2} S_{x_1} \Delta S \frac{\partial}{\partial x_1} \zeta_n dx\\
    & \ =-2\varepsilon \int_{\mathbb R^2} S_{x_1x_1} \Delta S\zeta_n dx + 2\varepsilon\int_{\mathbb R^2} \nabla S_{x_1} \cdot\nabla S\frac{\partial \zeta_n}{\partial x_1} dx +2\varepsilon \int_{\mathbb R^2} S_{x_1} \nabla S\cdot \frac{\partial}{\partial x_1} \nabla \zeta_n dx\\
   &  =-2\varepsilon\int_{\mathbb R^2} S_{x_1x_1}\Delta S\zeta_n dx -\varepsilon \int_{\mathbb R^2}|\nabla S| ^2 \frac{\partial^2 \zeta_n}{\partial x_1^2} dx + 2\varepsilon \int_{\mathbb R^2} S_{x_1} \nabla S\cdot \frac{\partial}{\partial x_1} \nabla \zeta_n dx.
\end{align*}
For the second term on the right-hand side of \eqref{eq:H1_1}, we have the estimate
\begin{align*}
    & 2\int_{\mathbb R^2} S_{x_1} \left(\operatorname{div}\left(S\nabla V_S\right)\right)_{x_1} \zeta_n dx =2 \int_{\mathbb R^2} S_{x_1} \left(\left(\nabla S\cdot \nabla V_S \right)_{x_1} +\left(S\Delta V_S\right)_{x_1}\right) \zeta_n dx\\
& \ =2 \int_{\mathbb R^2} S_{x_1}\left(\nabla S\right)_{x_1}\cdot \nabla V_S \zeta_n dx + 2 \int_{\mathbb R^2} S_{x_1} \nabla S \cdot\left(\nabla V_S \right)_{x_1} \zeta_n dx\\
& \ +2\int_{\mathbb R^2} (S_{x_1})^2 \Delta V_S \zeta_n dx +2\int_{\mathbb R^2} S_{x_1} S \left(\Delta V_S\right)_{x_1}\zeta_n dx\\
& \ =-\int_{\mathbb R^2} (S_{x_1})^2 \Delta V_S \zeta_n dx - \int_{\mathbb R^2} (S_{x_1})^2 \nabla V_S\cdot \nabla \zeta_n dx\\
& \ +2\int_{\mathbb R^2} (S_{x_1})^2 \left(V_S\right)_{x_1x_1}\zeta_n dx + \int_{\mathbb R^2} S_{x_1} S_{x_2} \left(V_S\right)_{x_1 x_2} \zeta_n dx\\
& \ +2 \int_{\mathbb R^2} (S_{x_1})^2 \Delta V_S \zeta_n dx + 2 \int_{\mathbb R^2} S_{x_1} S\left(\Delta V_S \right)_{x_1} \zeta_n dx\,.
\end{align*}
For the last term above we use Lemma \eqref{eq:est_kernels2} to obtain the estimate
\begin{align*}
    & 2 \int_{\mathbb R^2} S_{x_1} S\left(\Delta V_S \right)_{x_1} \zeta_n dx\leq \|S\|_{L^\infty}\left(\int_{\R^2}S_{x_1}^2 dx + \int_{\R^d}\left|(\Delta V_S)_{x_1}\right|^2dx\right)\\
    & \ \leq \|S\|_{L^\infty}\left(\int_{\R^2}S_{x_1}^2 dx + C\left(\|\nabla S\|_{L^2(\R^2)}^2+\|\nabla I\|_{L^2(\R^2)}^2+\|\nabla R\|_{L^2(\R^2)}^2\right)\right)\,.
\end{align*}
The remaining terms may be treated estimating the terms involving the kernel $V_S$ in $L^\infty$ via \eqref{eq:est_kernels}. Hence, via the global estimates in Proposition \ref{propLp} with $p=+\infty$, for a suitable constant $C$ depending on the initial data and on $T$ (defined for all $T\geq 0$), we obtain
\[
2 \int_{\mathbb R^2} S_{x_1} S\left(\Delta V_S \right)_{x_1} \zeta_n dx\leq C\left(\|\nabla S\|_{L^2(\R^2)}^2+\|\nabla I\|_{L^2(\R^2)}^2+\|\nabla R\|_{L^2(\R^2)}^2\right)\,.
\]
 Here, we have used Young's inequality and the estimates \eqref{eq:cutoff_est} on the cut-off.
As for the third term on the right-hand side of \eqref{eq:H1_1}, we have the estimate
\begin{align*}
    &
-2\beta \int_{\mathbb R^2}S_{x_1} \left(SI\right)_{x_1} \zeta_n dx = -2 \beta \int_{\mathbb R^2}(S_{x_1})^2 I \zeta_n dx - 2\beta \int_{\mathbb R^2}S S_{x_1} I_{x_1}\zeta_n dx\\
& \ \leq 2\beta \|S\|_{L^{\infty}} \int_{\mathbb R^2}S_{x_1} I_{x_1} \zeta_n dx \leq C \int_{\mathbb R^2} (S_{x_1})^2 dx + C \int_{\mathbb R^2}(I_{x_1}) ^2 dx\,,
\end{align*}
for a suitable constant $C$ which is once again possibly depending on $T$ with $t\in [0,T]$. Combining the previous three estimates in \eqref{eq:H1_1}, we obtain
\begin{align}
     & \frac{d}{dt} \int_{\mathbb R^2} (S_{x_1}) ^2 \zeta_n dx \leq -2\varepsilon \int_{\mathbb R^2} S_{x_1x_1}\Delta S\zeta_n dx -\varepsilon \int_{\mathbb R^2}|\nabla S| ^2 \frac{\partial^2}{\partial x_1^2} \zeta_n dx + 2\varepsilon \int_{\mathbb R^2} S_{x_1} \nabla S\cdot \frac{\partial}{\partial x_1} \nabla \zeta_n dx\nonumber\\
    & \ + C\left(\|\nabla S\|_{L^2(\R^2)}^2+\|\nabla I\|_{L^2(\R^2)}^2+\|\nabla R\|_{L^2(\R^2)}^2\right),\label{eq:est_Sx1}
\end{align}
for a suitable constant $C\geq 0$ independent of $\varepsilon$, possibly depending on $T$ but globally defined for all $T\geq 0$. A similar estimate may be performed for the $x_2$ derivative of $S$, which yields
\begin{align}
    &  \frac{d}{dt} \int_{\mathbb R^2} (S_{x_2}) ^2 \zeta_n dx\leq -2\varepsilon\int_{\mathbb R^2} S_{x_2x_2}\Delta S \zeta_n dx -\varepsilon \int_{\mathbb R^2}|\nabla S| ^2 \frac{\partial^2}{\partial x_2^2} \zeta_n dx + 2\varepsilon \int_{\mathbb R^2} S_{x_2} \nabla S\cdot \frac{\partial}{\partial x_2} \nabla \zeta_n dx\nonumber\\
   & \  + C\left(\|\nabla S\|_{L^2(\R^2)}^2+\|\nabla I\|_{L^2(\R^2)}^2+\|\nabla R\|_{L^2(\R^2)}^2\right)\,.\label{eq:est_Sx2}
\end{align}
Combining \eqref{eq:est_Sx1} and \eqref{eq:est_Sx2}, we obtain
\begin{align*}
   &  \frac{d}{dt} \int_{\mathbb R^2} |\nabla S| ^2 \zeta_n dx \leq -2\varepsilon\int_{\mathbb R^2} (\Delta S) ^2 \zeta_n dx -\varepsilon \int_{\mathbb R^2} |\nabla S| ^2 \Delta \zeta_n dx \\
    & \ + 2\varepsilon \int_{\mathbb R^2}\left(S_{x_1}\nabla S  \frac{\partial }{\partial x_1}\nabla \zeta_n +S_{x_2}\nabla S  \frac{\partial }{\partial x_2}\nabla \zeta_n\right) dx
    + C\left(\|\nabla S\|_{L^2(\R^2)}^2+\|\nabla I\|_{L^2(\R^2)}^2+\|\nabla R\|_{L^2(\R^2)}^2\right)\\
    & \ \leq 2C\left(\|\nabla S\|_{L^2(\R^2)}^2+\|\nabla I\|_{L^2(\R^2)}^2+\|\nabla R\|_{L^2(\R^2)}^2\right),
\end{align*}
for a suitable constant $C\geq 0$ independent of $\varepsilon$ (for small $\varepsilon$), possibly depending on $T$ and globally defined for all $T\geq 0$. We follow the same method to estimate the $L^2$ norm of $\nabla I$. We start by computing
\begin{align}
    & \frac{d}{dt} \int_{\mathbb R^2} (I_{x_1}) ^2 \zeta_n dx = 2 \int _{\mathbb R ^2} I_{x_1} \left(I_{x_1}\right)_t \zeta_n dx\nonumber\\
& \ =2 \int_{\mathbb R ^2} I_{x_1}\left(\varepsilon \Delta I+\operatorname{div}\left(I\nabla V_I[S,I,R]\right) +\beta S I -\alpha I\right)_{x_1} \zeta_n dx\nonumber\\
& \ =2\varepsilon \int_{\mathbb R^2} I_{x_1}\left(\Delta I\right)_{x_1} \zeta_n dx +2\int_{\mathbb R^2} I_{x_1}\left(\operatorname{div}\left(I\nabla V_I\right)\right)\zeta_n dx+2 \int_{\mathbb R^2} I_{x_1}(\beta SI-\alpha I )_{x_1} \zeta_n dx.\label{eq:H1_2}
\end{align}
The first two terms on the right-hand side of \eqref{eq:H1_2} are computed exactly the same way as the first two terms on the right-hand side of \eqref{eq:H1_1}, namely
\begin{align*}
    &   2\varepsilon \int_{\mathbb R^2} I_{x_1}\left(\Delta I\right)_{x_1} \zeta_n dx=-2\varepsilon\int_{\mathbb R^2} I_{x_1x_1}\Delta I\zeta_n dx -\varepsilon \int_{\mathbb R^2}|\nabla I| ^2 \frac{\partial^2}{\partial x_1^2} \zeta_n dx + 2\varepsilon \int_{\mathbb R^2} I_{x_1} \nabla I\cdot \frac{\partial}{\partial x_1} \nabla \zeta_n dx,
\end{align*}
and
\begin{align*}
    & 2\int_{\mathbb R^2} I_{x_1}\left(\operatorname{div}\left(I\nabla V_I\right)\right)\zeta_n dx \leq C\left(\|\nabla S\|_{L^2(\R^2)}^2+\|\nabla I\|_{L^2(\R^2)}^2+\|\nabla R\|_{L^2(\R^2)}^2\right),
\end{align*}
for a suitable constant independent of $\varepsilon$ and possibly depending on $T$ with $t\in [0,T]$. For the final term of \eqref{eq:H1_2} we have
\begin{align*}
    &  2 \int_{\mathbb R^2} I_{x_1}(\beta SI-\alpha I )_{x_1} \zeta_n dx = 2\beta \int_{\mathbb R^2} S (I_{x_1})^2 \zeta_n dx +2\beta \int_{\mathbb R^2}I I_{x_1} S_{x_1} \zeta_n dx -2\alpha \int_{\mathbb R^2} (I_{x_1}) ^2 \zeta_n dx \\
    & \ \leq 2\beta \|S\|_{L^{\infty}} \int_{\mathbb R^2} (I_{x_1})^2 \zeta_n dx + \beta \|I\|_{L^{\infty}}
    \int_{\mathbb R^2} ((S_{x_1}) ^2+(I_{x_1})^2) \zeta_n dx\\
    & \ \leq  C\left(\|\nabla S\|_{L^2(\R^2)}^2+\|\nabla I\|_{L^2(\R^2)}^2+\|\nabla R\|_{L^2(\R^2)}^2\right)\,,
\end{align*}
which implies
\begin{align*}
    & \frac{d}{dt} \int_{\mathbb R^2} (I_{x_1}) ^2 \zeta_n dx \leq -2\varepsilon\int_{\mathbb R^2} I_{x_1x_1}\Delta I\zeta_n dx -\varepsilon \int_{\mathbb R^2}|\nabla I| ^2 \frac{\partial^2}{\partial x_1^2} \zeta_n dx + 2\varepsilon \int_{\mathbb R^2} I_{x_1} \nabla I\cdot \frac{\partial}{\partial x_1} \nabla \zeta_n dx\\
    & \ +C\left(\|\nabla S\|_{L^2(\R^2)}^2+\|\nabla I\|_{L^2(\R^2)}^2+\|\nabla R\|_{L^2(\R^2)}^2\right).
\end{align*}
By a similar estimate on $I_{x_2}$, combined with the previous one, similarly to what done for $\nabla S$, we obtain
\begin{align*}
    &  \frac{d}{dt} \int_{\mathbb R^2} |\nabla I|^2 \zeta_n dx \leq -2\varepsilon\int_{\mathbb R^2} (\Delta I) ^2\zeta_n dx -\varepsilon \int_{\mathbb R^2}|\nabla I| ^2 \Delta \zeta_n dx \\
   & \ + 2\varepsilon \int_{\mathbb R^2}\left(I_{x_1} \nabla I\cdot \frac{\partial}{\partial x_1} \nabla \zeta_n+ I_{x_2} \nabla I\cdot \frac{\partial}{\partial x_2} \nabla \zeta_n\right) dx+  C\left(\|\nabla S\|_{L^2(\R^2)}^2+\|\nabla I\|_{L^2(\R^2)}^2+\|\nabla R\|_{L^2(\R^2)}^2\right)\\
   & \ \leq 2C\left(\|\nabla S\|_{L^2(\R^2)}^2+\|\nabla I\|_{L^2(\R^2)}^2+\|\nabla R\|_{L^2(\R^2)}^2\right)\,.
\end{align*}
A similar estimate may be obtained for  $\nabla R$ (we omit the details). Combining all the above, we obtain
\begin{align*}
    &  \frac{d}{dt} \int_{\mathbb R^2} \left(|\nabla S|^2+|\nabla I|^2+ |\nabla R|^2\right) \zeta_n dx \leq C\left( \int_{\mathbb R^2} |\nabla S|^2 dx +\int_{\mathbb R^2} |\nabla I|^2 dx+\int_{\mathbb R^2} |\nabla R|^2 dx\right).
\end{align*}
Here, the constant $C$ depends on the initial conditions, and it involves $L^\infty$ estimates of $(S,I,R)$ on $[0,T]$.
Integrating in time and using monotone convergence to let $n\rightarrow+\infty$, we get the desired assertion.
\end{proof}

\subsection{Vanishing viscosity limit}\label{subsec:vanishing}

The goal of this subsection is to let $\varepsilon \searrow 0$ in \eqref{mainsystem} in order to detect the existence of a solution to \eqref{mainsystem_intro}.

\begin{Proposition}
Let $(S_{\varepsilon}, I_{\varepsilon}, R_{\varepsilon})$ be the unique solution to the system \eqref{mainsystem} with the diffusion constant $\varepsilon>0$. Under the assumptions of Theorem \ref{thm:main}, there exists a solution
$$(S,I,R)\in L^\infty ([0,T];H^1(\mathbb R ^2))^3$$ that solves \eqref{mainsystem_intro} which is the $\varepsilon\searrow 0$ limit of a subsequence of $(S_{\varepsilon}, I_{\varepsilon}, R_{\varepsilon})$ in the $L^2(\R^2\times [0,T])$-sense for all $T\geq 0$.
\end{Proposition}

\begin{proof}
Consider the diffusion constant to be $\varepsilon_m= \frac{1}{m}, m\in \mathbb N$ and $(S_m, I_m, R_m)$ to be corresponding solution to the system \eqref{mainsystem} with the diffusion constant $\varepsilon_m$. From Propositions \ref{prop:H1} and \ref{prop:H1}, we know that $S_m, I_m, R_m\in L^\infty([0,T],H^1(\mathbb R ^2))$. Consider not a bounded domain $\Omega \subset \R^2$. Since $S_m$, $I_m$, and $R_m$ are uniformly bounded in $H^1(\Omega)$ on $[0,T]$, then $\Delta S_m, \Delta I_m, \Delta R_m$ are uniformly bounded in $H^{-1}(\Omega)$ on $[0,T]$. Moreover, the terms $\beta SI$ and $\alpha I$ belong to $H^1(\Omega)$ due to Proposition \ref{propLp} with $p=+\infty$ and Proposition \ref{prop:H1}. Furthermore, we claim that the terms $\operatorname{div}(S\nabla V_S)$, $\operatorname{div}(I\nabla V_I)$, $\operatorname{div}(R\nabla V_R)$ belong to $L^2(\Omega)$. In order to see this, write for example
\[\mathrm{div}(R\nabla V_R)=\nabla R\cdot \nabla V_R + R\Delta V_R\,.\]
Since $\nabla R$ is uniformly bounded in $H^1$ due to Proposition \ref{prop:H1} and $\nabla V_R$ is uniformly bounded in $L^\infty$ due to Lemma \ref{lem:nonlocal_etsimates}, the first term in the above sum is uniformly bounded in $L^2$. A similar conclusion holds for the second term still due to Lemma \ref{lem:nonlocal_etsimates}. Similar considerations hold for the other terms in divergence form involving the nonlocal interaction operators.
Hence, by considering each equation in \eqref{mainsystem}, we can conclude that $S_t$, $I_t$, and $R_t$ all belong to $L^\infty([0,T]\,;\,\,H^{-1}(\Omega))$. Hence, we have
\begin{equation}\label{unifbound4seq}
\begin{array}{cc}
    S_m, I_m, R_m \quad\hbox{uniformly bounded in}\,\, L^2 ([0,T]; H^1(\mathbb R^2)),\\
    S_{m,t}, I_{m,t}, R_{m,t} \quad\hbox{uniformly bounded in}\,\, L^2 ([0,T]; H^{-1}(\mathbb R^2)).
\end{array}
\end{equation}
We may therefore use the classical Aubin-Lion lemma to obtain strong compactness of $$\{(S_m,I_m,R_m)\}_{m=1} ^{\infty}$$ in $L^2([0,T]\times \Omega$ for every bounded domain $\Omega$, and then, by a standard diagonal procedure, the existence of a subsequence of $\{(S_m,I_m,R_m)\}_{m=1} ^{\infty}$ which converges almost everywhere on $\R^2\times [0,T]$ to some limit function
$(S,I,R) \in L^2([0,T]\times \mathbb R ^2)$ as $m\rightarrow \infty$.
We claim this limit point is a weak solution to the system \eqref{mainsystem_intro}. To prove this claim, consider the test function $\phi \in C_c^\infty ([0,T]\times \mathbb R^2)$. For any $k\in \mathbb N$ we have:
\begin{align*}
    & 0 =  \int_0^t\int_{\mathbb R^2} \left(S_{k,t} \phi-\varepsilon_k \Delta S_k \phi -\operatorname{div}(S_k \nabla V_S)\phi +\beta S_kI_k\phi \right)dxdt\\
    & \ =-\int_0^t \int_{\mathbb R^2}S_k \phi_t dxdt - \varepsilon_k \int_0^t \int_{\mathbb R^2} S_k \Delta \phi dxdt + \int_0^t \int_{\mathbb R^2} S_k\nabla \phi\cdot\nabla V_S dxdt +\beta \int_0^t \int_{\mathbb R^2} S_k I_k\phi dxdt\,.
\end{align*}
The regularity of $\phi$ implies $\Delta \phi \in L^2_{x,t}$, which implies, due to Proposition \ref{propLp}, that
$$-\varepsilon_k\int_0^t \int_{\mathbb R^2} S_k \Delta \phi dxdt$$ is converging to zero as $k\rightarrow \infty$ strongly in $L^2$. The regularity of the kernels \eqref{eq:ass_reg_W} imply
\[
\nabla V_S[S_k,I_k,R_k] =\nabla W_{SS}* S_k+\nabla W_{SI}* I_k+\nabla W_{SR}*R_k\rightarrow \nabla V_S[S,I,R]\quad    a.e. \text{ as } n\rightarrow \infty,
\]
which implies
\[
\int_0^t \int_{\mathbb R^2} S_k\nabla V_S[S_k,I_k,R_k] \cdot \nabla \phi dx dt \rightarrow \int_0^t \int_{\mathbb R^2} S\nabla V_S[S,I,R] \cdot \nabla \phi dx dt
\]
as $k\rightarrow +\infty$. The reaction term in the $S$ equation also easily passes to the limit due to strong $L^2$ convergence, namely
\[
\beta\int_0^t \int _{\mathbb R^2} S_k I_k \phi dxdt \rightarrow \beta\int_0^t \int _{\mathbb R^2} S I \phi dxdt\,.
\]
Hence, we obtain
\[
\int_0^t \int _{\mathbb R^2} (S_t +\beta SI -\operatorname{div}(S \nabla V_S[S,I,R]))\phi dx dt =0\,.
\]
Using the same method, we can show that $(S,I,R)$ solve the other two constraints of the system \ref{mainsystem_intro} and hence, it is a weak solution to the system.
The sequence $\{(S_k,I_k,R_k)\}_{k=1}^\infty$ also satisfies \eqref{unifbound4seq}. Hence, a standard weak-lower semi-continuity argument implies
\begin{align*}
    & \|S\|_{L^2([0,T];H^1(\mathbb R ^2)}\leq \liminf_{n\rightarrow \infty} \|S_n\|_{L^2([0,T];H^1(\mathbb R ^2)},\\
& \|I\|_{L^2([0,T];H^1(\mathbb R ^2)}\leq \liminf_{n\rightarrow \infty} \|I_n\|_{L^2([0,T];H^1(\mathbb R ^2)},\\
& \|R\|_{L^2([0,T];H^1(\mathbb R ^2)}\leq \liminf_{n\rightarrow \infty} \|R_n\|_{L^2([0,T];H^1(\mathbb R ^2)}\,,
\end{align*}
which concludes the proof.
\end{proof}

\section{Uniqueness}
\label{sec:unique}

In order to conclude the proof of Theorem \ref{thm:main}, we only need to prove the uniqueness statement, which we rephrase in the next Proposition.

\begin{Proposition}\label{uniqueness}
 Under the same assumptions of Theorem \ref{thm:main}, for a fixed $T\geq 0$, there exists at most one weak solution to \eqref{mainsystem_intro} with $L^\infty([0,T]\,;\,\, L^1(\R^2)\cap L^\infty(\R^2)\cap H^1(\R^2))$ regularity with a given initial condition $(S_0, I_0, R_0)$ satisfying the assumptions of Theorem \ref{thm:main}.
\end{Proposition}

\begin{proof}
Let $(S_1,I_1,R_1)$ and $(S_2,I_2,R_2)$ be two solutions to the system \eqref{mainsystem_intro} in $L^\infty([0,T];L^1(\R^2)\cap L^\infty(\R^2)\cap H^1(\R^2))$, with $S_1(0)=S_2(0)=S_0$, $I_1(0)=I_2(0)=I_0$ and $R_1(0)=R_2(0)=R_0$. We compute
\begin{align}
    & \frac{d}{dt} \int_{\mathbb R^2 } (S_1-S_2)^2 \zeta_n dx = 2\int_{\mathbb R^2} (S_1-S_2)(S_1-S_2)_t \zeta_n dx\nonumber\\
    & \  = 2 \int_{\mathbb R^2} (S_1-S_2)(\operatorname{div}(S_1\nabla V_S[S_1,I_1,R_1])-\operatorname{div}(S_2\nabla V_S[S_2,I_2,R_2])-\beta S_1 I_1+\beta S_2 I_2)\zeta_n dx\nonumber\\
    & \ =2 \int_{\mathbb R^2} (S_1-S_2) \operatorname{div}(S_1 (\nabla V_S[S_1,I_1,R_1]-\nabla V_S[S_2,I_2,R_2])) \zeta_n dx\nonumber\\
   & \  + 2 \int_{\mathbb R^2} (S_1-S_2) \operatorname{div}((S_1-S_2)\nabla V_S[S_2,I_2,R_2]) \zeta_n dx \nonumber\\
    & \ -2\beta \int_{\mathbb R^2} (S_1-S_2) S_1 (I_1-I_2) \zeta_n dx -2\beta \int_{\mathbb R^2} (S_1-S_2)^2 I_2 \zeta_n dx\,.\label{eq:est_unique1}
\end{align}
We denote for simplicity $\mathcal{S}_1=(S_1,I_1,R_1)$ and $\mathcal{S}_2=(S_2,I_2,R_2)$. The first term on the right-had side of \eqref{eq:est_unique1} may be written as
\begin{align*}
    & 2 \int_{\mathbb R^2} (S_1-S_2) \operatorname{div}(S_1 (\nabla V_S[\mathcal{S}_1]-\nabla V_S[\mathcal{S}_2])) \zeta_n dx= 2\int_{\mathbb R^2} (S_1-S_2) \operatorname{div}(S_1 \nabla V_S[\mathcal{S}_1-\mathcal{S}_2]) \zeta_n dx\\
& \ =2 \int_{\mathbb R^2} (S_1-S_2) (\nabla S_1\cdot \nabla V_S[\mathcal{S}_1-\mathcal{S}_2]+S_1\Delta V_S[\mathcal{S}_1-\mathcal{S}_2]) \zeta_n dx\,.
\end{align*}
We apply the result in Lemma \ref{lem:nonlocal_etsimates} to get
\begin{align*}
    & \|\nabla V_s[\mathcal{S}_1-\mathcal{S}_2]\|_{L^\infty(\R^2)}+\|\Delta V_S[\mathcal{S}_1-\mathcal{S}_2]\|_{L^\infty(\R^2)}\\
    & \ \leq C\left(\|S_1-S_2\|_{L^2(\R^2)}+\|I_1-I_2\|_{L^2(\R^2)}+\|R_1-R_2\|_{L^2(\R^2)}\right)\,.
\end{align*}
Moreover, we know from Proposition \ref{prop:H1} that there exists a constant $C$ possibly depending on $T\geq 0$ such that
\[\sup_{0\leq t\leq T}\left(\|\nabla S_1\|_{L^2(\R^2)} + \|S_1\|_{L^2(\R^2)}\right)\leq C\,.\]
Hence, we may use Hoelder's inequality and the fact that $\zeta_n\in [0,1]$ to obtain
\begin{align*}
    &  2 \int_{\mathbb R^2} (S_1-S_2) (\nabla S_1\cdot \nabla V_S[\mathcal{S}_1-\mathcal{S}_2]+S_1\Delta V_S[\mathcal{S}_1-\mathcal{S}_2]) \zeta_n dx\\
    & \ \leq 2\|S_1-S_2\|_{L^2} \|\nabla S_1\|_{L^2} \|\nabla V_S[\mathcal{S}_1-\mathcal{S}_2]\|_{L^\infty}+2\|S_1-S_2\|_{L^2} \|S_1\|_{L^2} \|\Delta V_S[\mathcal{S}_1-\mathcal{S}_2]\|_{L^\infty}\\
    & \ \leq C \|S_1-S_2\|_{L^2}(\|S_1\|_{L^2}+\|\nabla S_1\|_{L^2})(\|S_1-S_2\|_{L^2} + \|I_1-I_2\|_{L^2} + \|R_1-R_2\|_{L^2})\\
    & \ \leq C \|S_1-S_2\|_{L^2}^2 + C \|I_1-I_2\|_{L^2}^2 + C \|R_1-R_2\|_{L^2}^2.
\end{align*}
We estimate the second term on the right-hand side of \eqref{eq:est_unique1} very similarly as follows:
\begin{align*}
    & 2 \int_{\mathbb R^2} (S_1-S_2) \operatorname{div}((S_1-S_2)\nabla V_S[\mathcal{S}_2]) \zeta_n dx\\
    & \  = 2 \int_{\mathbb R^2} (S_1-S_2) \nabla (S_1-S_2)\cdot \nabla V_S[\mathcal{S}_2] \zeta_n dx + 2 \int_{\mathbb R^2} (S_1-S_2)^2 \Delta V_S[\mathcal{S}_2] \zeta_n dx\\
    & \  = - \int_{\mathbb R ^2} (S_1-S_2)^2 ( \Delta V_S[\mathcal{S}_2] \zeta_n + \nabla V_S[\mathcal{S}_2] \cdot \nabla \zeta _n) dx\\
    & \  + 2 \int_{\mathbb R^2} (S_1-S_2)^2 \Delta V_S[\mathcal{S}_2] \zeta_n dx\\
    & \  \leq C \left(\|\Delta V_S[\mathcal{S}_2]\|_{L^\infty} + \|\nabla V_S[\mathcal{S}_2]\|_{L^\infty}\right) \|S_1-S_2\|_{L^2}^2\leq C \|S_1-S_2\|_{L^2}^2 .
\end{align*}
To control the last two terms in \eqref{eq:est_unique1}, namely the reaction terms, we use the estimate
\begin{align*}
    & +2\beta \int_{\mathbb R^2} (S_1-S_2) S_1 (I_1-I_2) \zeta_n dx + 2\beta \int_{\mathbb R^2} (S_1-S_2)^2 I_2 \zeta_n dx\\
    & \leq 2 \beta  \|S_1\|_{L^\infty} (\|S_1-S_2\|_{L^2}^2+\|I_1-I_2\|_{L^2}^2) + 2\beta \|I_2\|_{L^\infty} \|S_1-S_2\|_{L^2}^2\\
    & \leq C\|S_1-S_2\|_{L^2}^2 +C\|I_1-I_2\|_{L^2}^2\,.
\end{align*}
Combining the above estimates in \eqref{eq:est_unique1}, we obtain
\begin{equation}\label{eq:est_unique_2}
    \frac{d}{dt} \int_{\mathbb R^2 } (S_1-S_2)^2 \zeta_n dx \leq C \left(\|S_1-S_2\|_{L^2}^2 +  \|I_1-I_2\|_{L^2}^2 +  \|R_1-R_2\|_{L^2}^2\right),
\end{equation}
for a suitable constant $C$ depending on the initial conditions and possibly on $T$. The estimate \eqref{eq:est_unique_2} may be easily performed on the $I$ and $R$ components as well. The divergence terms may be dealt with in exactly the same way. The reaction terms are only slightly different by they give rise to similar terms. Combining the three estimates, we obtain
\begin{align*}
    & \frac{d}{dt} \int_{\mathbb R^2 } ((S_1-S_2)^2 + (I_1-I_2)^2 + (R_1-R_2)^2) \zeta_n dx \\
    & \ \leq C \int_{\mathbb R^2 } ((S_1-S_2)^2 + (I_1-I_2)^2 + (R_1-R_2)^2) dx.
\end{align*}
By integrating in time and by letting $n\rightarrow+\infty$, using Gronwall's inequality, we finally obtain
\begin{align*}
    & \int_{\mathbb R^2 } ((S_1-S_2)^2 + (I_1-I_2)^2 + (R_1-R_2)^2) dx\\
   & \  \leq e^{Ct} \int_{\mathbb R^2 } ((S_1(0)-S_2(0))^2 + (I_1(0)-I_2(0))^2 + (R_1(0)-R_2(0))^2) dx = 0\,,
\end{align*}
which implies $(S_1,I_1,R_1)=(S_2,I_2,R_2)$ almost everywhere on $\R^2\times [0,T]$.
\end{proof}

\section{Steady states}\label{sec:steady}

\subsection{A general discussion}\label{subsec:discussion}
After having solved system \eqref{mainsystem_intro}, we now turn our attention to the existence or non-existence of steady states. First of all, we recall from classical SIR models theory (see e. g. \cite{martcheva}) that the homogeneous SIR model \eqref{eq:SIR} dos not feature any steady state other than the so-called "disease free" equilibrium, in which the infectious compartments equals zero and $S$ stabilises at some positive value $S_\infty>0$. Since the total population $N=S+I+R$ is, in this case, preserved in time, the recovered compartment also stabilises to $R_\infty=N-S_\infty$. In this sense, the classical $SIR$ model lacks of asymptotic complexity, in that the only expected long time behavior is the one in which the infectious compartment is cleared in infinite time. One of the many variants of the classical SIR model is the one in which immunity is loss immediately after having recovered from the disease. In this case, the recovered compartment is not considered, and individuals move from $I$ back to $S$ after recovery. The $SIS$ (well-known) model looks like
\begin{equation}\label{eq:SI}
    \begin{dcases}
        \dot{S}=-\beta S I+\alpha I,& \\
        \dot{I}=\beta SI -\alpha I\,. &
    \end{dcases}
\end{equation}
Having defined the \emph{basic reproduction number}
\[\mathcal{R}_0=\frac{\beta N}{\alpha}\,,\]
it is well known that all solutions to \eqref{eq:SI} converge as $t\rightarrow+\infty$ to the \emph{disease free equilibrium} $(S,I)=(N,0)$ in case $\mathcal{R}_0<1$, whereas they converge towards the \emph{endemic equilibrium} $(S,I)=\left(\frac{\alpha}{\beta},N-\frac{\alpha}{\beta}\right)$. We recall that the total population $N=S+I$ is conserved in time.
The example of model \eqref{eq:SI} is paradigmatic of a huge variety of situations in epidemiological models, in which more compartments may be added (exposed, asymptomatic, quarantined, vaccinated, etc., not to mention models with multiple strains or with more than one susceptible classe), in which the transition terms may be complemented with migration terms and natural death terms, and in which the incidence terms (those responsible for transitions to infectious compartments) may be normalised by the total populations. In most of said models, a basic reproduction number similar to the above $\mathcal{R}_0$ may be defined which acts as a threshold with respect to the long time behavior. Typically, for $\mathcal{R}_0<1$ solutions stabilise around a disease free equilibrium, whereas for $\mathcal{R}_0>1$ solutions converge towards an endemic equilibrium with a non trivial value for the infectious compartment.

The presence of spatially depending nonlocal terms in the spirit of \eqref{mainsystem_intro} makes the situation quite different. First of all, with the extent of the generality of \eqref{mainsystem_intro}, the total population density $N(x,t)=S(x,t)+I(x,t)+R(x,t)$ is not "pointwise" preserved. In a situation such as \eqref{mainsystem}, where all "reaction" terms in the system balance, the total population
\[\int_\R N(x,t) dx\]
is indeed preserved in time, as we proved in Lemma \ref{lem:conservation}. As for the density $N(x,t)$ of the total population, a significant case in which one may decouple $N$ from the rest of the system is the one in which all interaction kernels coincide. In the case of the SIR model \eqref{mainsystem_intro} this means
\begin{equation}\label{eq:W_allthesame}
    W_{\xi\eta}=W\qquad \hbox{for all $\xi,\eta\in \{S,I,R\}$}\,.
\end{equation}
In that case, by adding up the three equations in \eqref{mainsystem_intro}, we get the decoupled equation
\begin{equation}\label{eq:N}
    N_t = \mathrm{div}(N \nabla W\ast N)\,.
\end{equation}
As mentioned in the introduction, we stress that there are valid reasons to consider cases in which \eqref{eq:W_allthesame} is not satisfied, for example if one wants to temper the mutual attraction between individuals according to their compartments (e.g. by assuming that individuals of the $I$ class are subject to a wider repulsive range compared to other compartments). The study of the existence of possible steady states becomes in this case more challenging and will be postponed to future studies.

In this section we shall restrict to the case in which \eqref{eq:W_allthesame} is satisfied. As mentioned in the introduction, the interaction potential $W$ is responsible for attractive and/or repulsive drifts among individuals. when the potential $W$ is smooth and attractive, namely $W(x)=w(|x|)$ with $w$ increasing on $[0,+\infty)$, the solution $N$ to \eqref{eq:N} tends to concentrate to a Dirac delta with total mass given by the mass of the initial condition $N_0$, centered at the center of mass $N_0$. In this case, the only possible steady state for $N$ would be a Dirac delta, which is incompatible with the other terms of the system \eqref{mainsystem_intro}. If, on the other hand, $W$ is repulsive, that is $w$ decreases on $[0,+\infty)$, then the solution $N$ is expected to spread on $\R^2$, possibly reaching the zero state in infinite time. This situation would be incompatible with the existence of any steady state other than zero. A case of particular interest, which often gives rise to the existence of nontrivial steady states, is the one in which $w$ decreases on some interval $[0,\lambda)$ and increases on $(\lambda+\infty)$, that is the so-called \emph{repulsive-attractive} case. We shall deal with this case in the next subsection and provide, at least in a very specific case for $W$, a way to detect steady states.

\subsection{A specific repulsive-attractive potential}\label{subsec:example}
We restrict for simplicity to the one-space dimensional case and we set
\begin{equation}\label{eq:W_rep_att}
    W(x)= x^2-\gamma|x|
\end{equation}
for some constant $\gamma>0$. The above potential is attractive for $|x|\geq \gamma/2$ and repulsive for $|x|\leq \gamma/2$. We consider then the one-dimensional SIS model
\begin{align}
   & \partial_t S = \partial_x (S W'*(S+I)) -\beta SI +\alpha I \,,\nonumber \\
      & \partial_t I = \partial_x (I W'*(S+I)) +\beta SI -\alpha I\,.\label{eq:SIS_space}
\end{align}
We are looking for steady states of \eqref{eq:SIS_space}, therefore we want to find $S=S(x)$ and $I=I(x)$ solving
\begin{align}
   & 0= \partial_x (S W'*(S+I)) -\beta SI +\alpha I\,, \nonumber \\
      & 0 = \partial_x (I W'*(S+I)) +\beta SI -\alpha I\,.\label{eq:SIS_stat}
\end{align}
We set $N=S+I$ and
\[M=\int_\R N(x) dx>0\,.\]
Since \eqref{eq:SIS_stat} is invariant by space translations, we also assume that the total population has zero center of mass
\[\int_\R y N(y) dy =0\,.\]
By taking the sum of the two equations in \eqref{eq:SIS_stat}, we get that $N$ satisfies the equations
\begin{equation}\label{eq:steady_nonloc}
    W'\ast N = 0\,,\qquad \hbox{on $\mathrm{supp}(N)$}\,.
\end{equation}
To solve \eqref{eq:steady_nonloc}, let us set
\[F(x)=\int_{-\infty}^x N(y) dy\,.\]
The equation \eqref{eq:steady_nonloc} may be re-written as
\begin{align*}
    & 0 = 2\int_\R (x-y)N(y) dy -\gamma\int_\R \mathrm{sign}(x-y) N(y) dy\\
    & \ = 2x M - \gamma\int_{-\infty}^x N(y) dy +\gamma\int_x^{+\infty} N(y) dy\\
    & \ = 2x M- \gamma F(x)+ \gamma M-\gamma F(x)=M(\gamma+2x)-2\gamma F(x)\,,
\end{align*}
which yields
\[F(x)=\frac{M(\gamma+2x)}{2\gamma}\,.\]
By differentiating with respect to $x$, we get
\[N(x)=\frac{M}{\gamma}\mathbf{1}_{[-\gamma/2,\gamma/2]}(x)\,.\]
Recalling
\[S(x)+I(x)=N(x)\,,\]
we assume both $S$ and $I$ are supported on the interval $[-\gamma/2,\gamma/2]$. Since $N$ is constant on said interval, we may assume that $S$ and $I$ are also constant and perform the classical
analysis of the steady states in \eqref{eq:SI} in order to detect these constants. We therefore set
\[S(x)=S_\infty \mathbf{1}_{[-\gamma/2,\gamma/2]}(x)\,,\qquad I(x)=I_\infty \mathbf{1}_{[-\gamma/2,\gamma/2]}(x),\]
and prescribe the stationary condition on the reaction terms
\[I_\infty(\beta S_\infty-\alpha)=0\,,\]
which yields the \emph{disease free} solution
\[S(x)= \frac{M}{\gamma}\mathbf{1}_{[-\gamma/2,\gamma/2]}(x)\,,\qquad I_\infty(x)=0\]
and the \emph{endemic equilibrium} solution
\[S_\infty(x)=\frac{\alpha}{\beta}\mathbf{1}_{[-\gamma/2,\gamma/2]}(x)\,,\qquad I_\infty(x)=\left(\frac{M}{\gamma}-\frac{\alpha}{\beta}\right)\mathbf{1}_{[-\gamma/2,\gamma/2]}(x)\,.\]
Clearly, the latter exists if and only if
\[\frac{M}{\gamma}>\frac{\alpha}{\beta}\,,\]
which is equivalent to requiring that the (space dependent) basic reproduction number
\begin{equation}\label{eq:new_R0}
    \mathcal{R}_0=\frac{M\beta}{\gamma\alpha}
\end{equation}
is strictly larger than $1$. We have therefore proven the following result.

\begin{Theorem}[Steady states]\label{thm:states}
    The stationary system \eqref{eq:SIS_stat} with $W$ given by \eqref{eq:W_rep_att} with $\gamma>0$ has the one-parameter family of solutions
    \[S(x)=\frac{M}{\gamma}\mathbf{1}_{[-\gamma/2,\gamma/2]}(x)\,,\qquad I(x)=0\]
    for all $M\geq 0$. Moreover, \eqref{eq:SIS_stat} has the additional one-parameter family of solutions
    \[S(x)=\frac{\alpha}{\beta}\mathbf{1}_{[-\gamma/2,\gamma/2]}(x)\,,\qquad I(x)=\left(\frac{M}{\gamma}-\frac{\alpha}{\beta}\right)\mathbf{1}_{[-\gamma/2,\gamma/2]}(x)\]
    for all
    \[M\in \left(\frac{\alpha \gamma}{\beta},+\infty\right)\,.\]
\end{Theorem}

\begin{Remark}
    \emph{We notice that our $\mathcal{R}_0$ in \eqref{eq:new_R0} involves the parameter $\gamma>0$, which is proportional to the repulsive range in the nonlocal interaction terms of \eqref{eq:SIS_space}. As it turns out from the condition $\mathcal{R}_0>1$, the existence of the endemic equilibrium is more likely to hold for small $\gamma$ and less likely to hold for large $\gamma$. This is not surprising. Indeed, large $\gamma$ means a larger repulsive range, that is, individuals are less likely to get attracted. This is somehow impacting the incidence. A large value of $\gamma$ somehow tempers the value of the incidence rate $\beta$ and reduces the chance of transmission of the disease from susceptible to infectious individuals.}
\end{Remark}

The method we illustrated in this subsection may be extended to more complicated models (possibly with more compartments) in which the mass of the total population is constant in time.
Indeed, when $W$ is given as in \eqref{eq:W_rep_att} we may obtain a constant total population on some interval, which allows us to consider the same constant states as in the corresponding homogeneous case.
\begin{figure}[hbt!]
    \centering
    \includegraphics[width=0.48\textwidth]{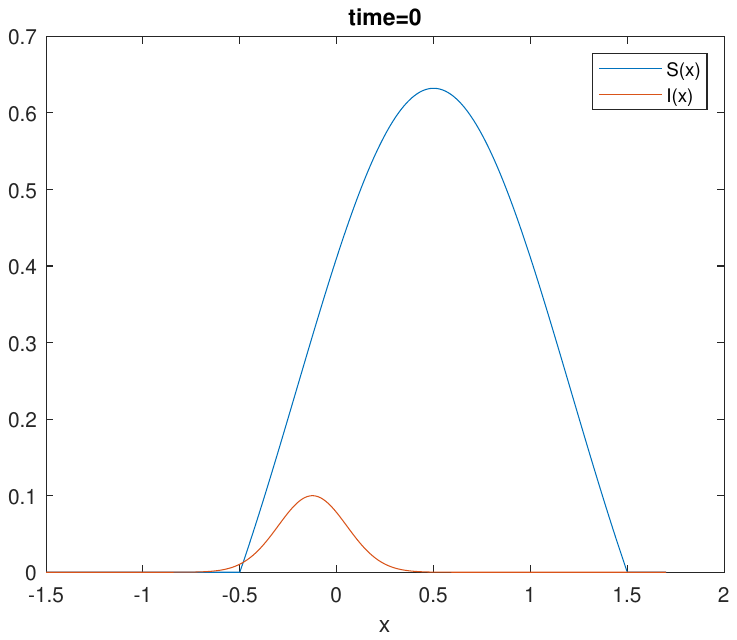}
    \includegraphics[width=0.48\textwidth]{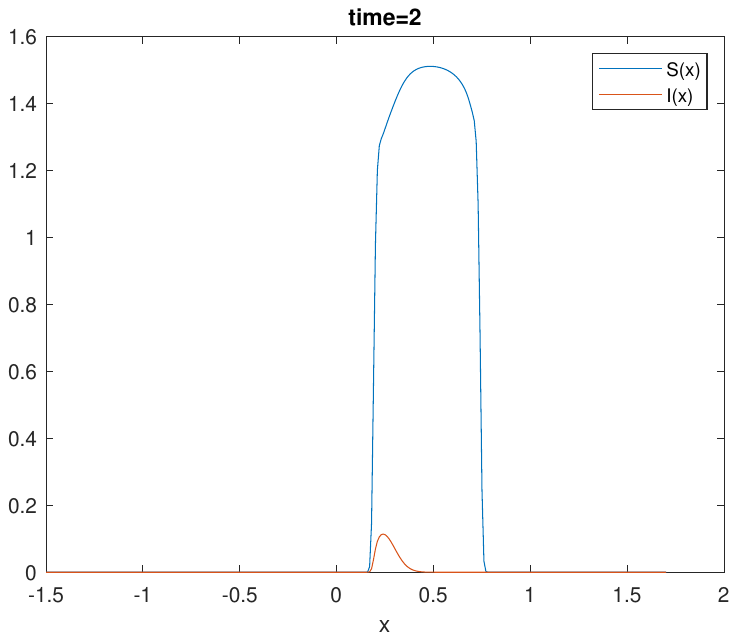}
    \includegraphics[width=0.48\textwidth]{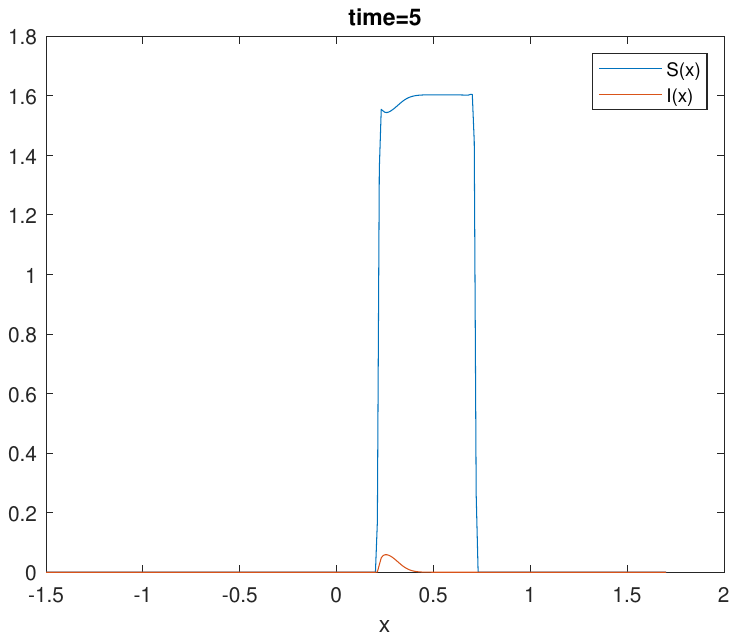}
    \includegraphics[width=0.48\textwidth]{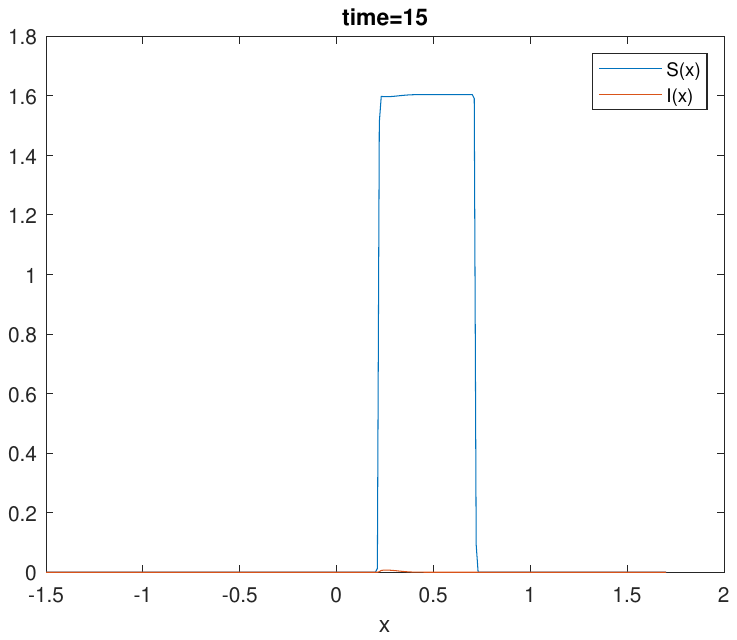}
    \caption{This figure shows the evolution of the epidemic in time, when $\mathcal{R}_0<1$, $\beta=0.5$. Here, the number of infected individuals is decreasing in time. Eventually, there are no infected individuals left, and we reach the disease-free equilibrium, and $S$ is supported in an interval of size $\gamma$.}
    \label{fig1}
\end{figure}

\begin{figure}[hbt!]
    \centering
    \includegraphics[width=0.48\textwidth]{t0.pdf}
    \includegraphics[width=0.48\textwidth]{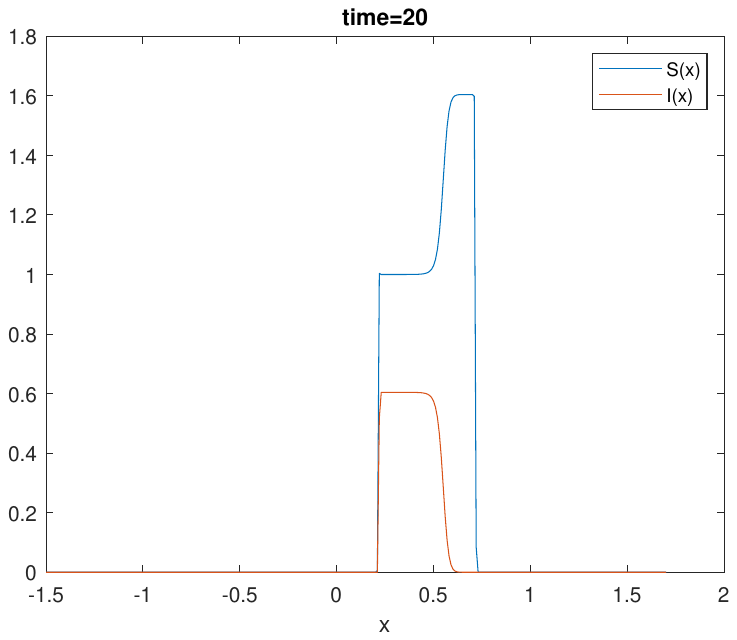}
    \includegraphics[width=0.48\textwidth]{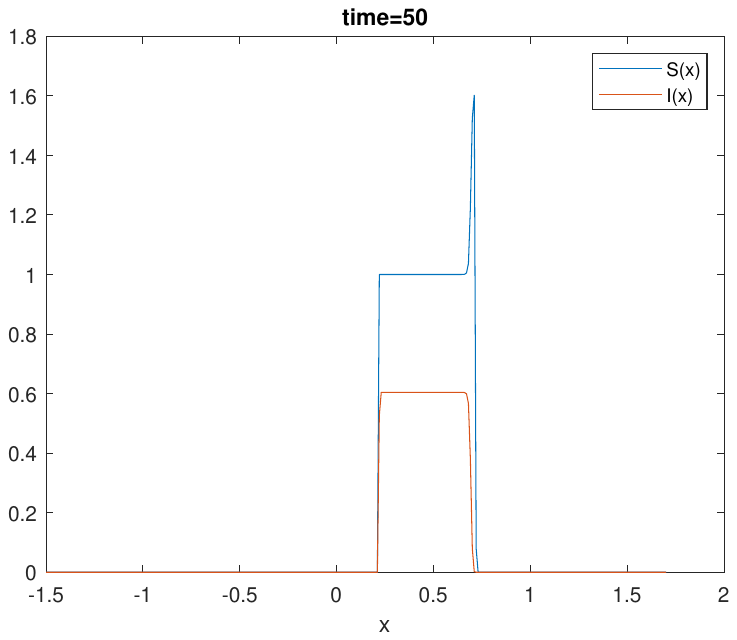}
    \includegraphics[width=0.48\textwidth]{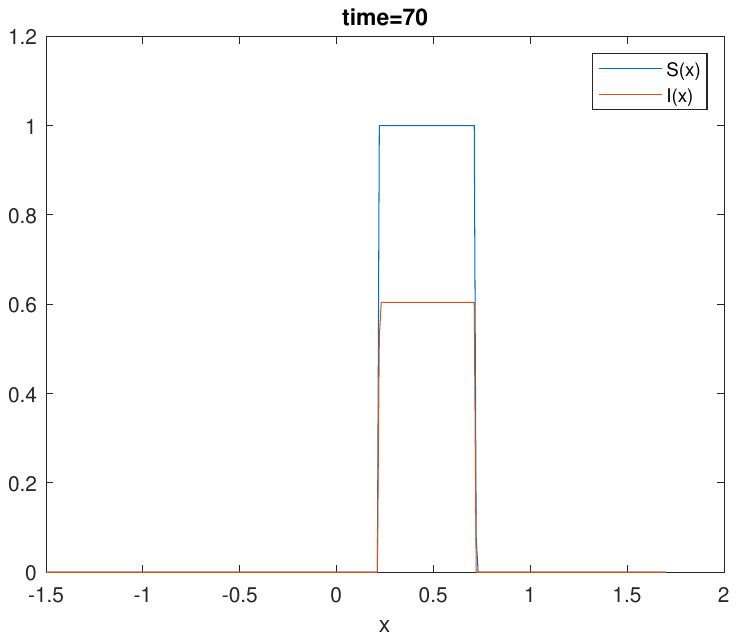}
    \caption{This Figure shows the evolution of the epidemic in time, when $\mathcal{R}_0>1$, $\beta=1$. Here, $I$ grows in time and spreads until both $S$ and $I$ are supported in the same interval of size $\gamma$, and they are both constant in this interval.}
    \label{fig2}
\end{figure}

\section{Numerical Simulations}\label{sec:numerical}

In this section, we provide some simple numerical simulations for the SIS model \eqref{eq:SIS_space}. We mainly consider two examples for this model, for which the former is converging to a disease-free equilibrium and the latter is converging to an endemic equilibrium. In these simulations, we are using the one-dimensional finite volume scheme introduced in \cite{Carrillo_Chertock} and Runge-Kutta method, considering the same initial values $S_0$ and $I_0$ where $I_0$ is supported in a relatively small interval with a smaller value in population density in comparison to $S_0$. In these simulations we are considering an equidistant discretization for space and time with the step size $\Delta x =0.01$ and $\Delta t=0.001$ and zero Dirichlet boundary condition. The domain considered for space is $[-1.7,1.7]$, while the time domain is as long as necessary for reaching the stationary state (for the first simulation $T= 15$ and the second simulation $T=70$). We set $\alpha=1$ and $\gamma = 0.5$. To create the two equilibria based on the analysis in the previous section we need respectively $\mathcal{R}_0>1$ and $\mathcal{R}_0<1$, where the value of $\beta$  is changing to satisfy these two inequalities. In Figure \ref{fig1}, we let  $\beta=0.5$ which leads to $\mathcal{R}_0<1$ and disease-free equilibrium, and in Figure \ref{fig2} we have $\beta=1$, which leads to the other inequality and endemic equilibrium. To further observe the behavior of the system, unlike the previous section, we do not assume the center of the mass to be zero. But regardless, we can see that in both cases, the population aggregation is such that soon both $S$ and $I$ are supported in an interval of length $\gamma$.




\clearpage
\section*{Acknowledgments}
MDF is partially supported by the Italian “National Centre for HPC, Big Data and Quantum Computing” - Spoke 5 “Environment and Natural Disasters” and by the Ministry of University and Research (MIUR) of Italy under the grant PRIN 2020- Project N. 20204NT8W4, Nonlinear Evolutions PDEs, fluid
dynamics and transport equations: theoretical foundations and applications. FGZ is supported by the PhD project “Non-local deterministic modelling for the diffusion of epidemics”, funded by the Italian PNRR (National Recovery and Resilience Plan), which is part of the Next Generation EU plan (DM
118/2023 M4C1 – Inv. 4.1 -
“Ricerca PNRR” - University of L'Aquila).
This research is also partially supported by the InterMaths Network, \url{www.intermaths.eu}.

\appendix
\section{}

Here we prove the results claimed in Remark \ref{Rem:localExt}.

\begin{Lemma}
There exists $T> 0$ such that there exists one and only one classical solution $u=(u_1,\ldots,u_N) \in (C^{1,2}_{t,x}(\mathbb R^d \times [0,T]))^N$ to the following system in $\mathbb R^d \times [0,T]$.
    \begin{equation}
        \begin{dcases}
            \partial_t u_i(x,t)=\varepsilon\Delta u_i + \mathrm{div}(u_i \nabla V_i[u_1,\ldots,u_N])+g_i(u_1,\ldots,u_N), & x\in \R^d\,,\,\, t>0\\
            u_i(x,0)=u_i^0(x), &
        \end{dcases}
    \end{equation}
    where $i=1,\ldots,N$, with nonlocal operators of the form
    \[V_i[u_1,\ldots,u_N](x)=\sum_{j=1}^N W_{ij}\ast u_i(x)\,,\]
    provided the interaction kernels $W_{ij}$ satisfy the same assumptions as in \eqref{eq:ass_reg_W}, the functions $g_i:\R ^N\rightarrow \R$ feature $C^1$ regularity, and $u_i^0\in L^1\cap L^\infty$.
\end{Lemma}

\begin{proof}
Similarly to the proof provided for Theorem \ref{prop:local}, we introduce $\mathcal{T}_1,\ldots \mathcal{T}_N$ which results from Duhamel's principle. To drive this results we need to prove that $\mathcal{T}_i$ is well-defined and a contraction for any $i\in \{1,\ldots,N\}$ in \begin{align*}
    & X_{r, T}:=\left\{\vphantom{\int}(u_1,\dots,u_N) \in C\left([0, T] ; \left(L^{1}\left(\mathbb{R}^{2}\right) \cap L^{\infty}\left(\mathbb{R}^{2}\right)\right)^N\right):\,\Sigma_{i=1}^N\vertiii{u_i}_{T}\leq r\right\}\,,
\end{align*}
Let $p=1,\infty$. Let $u=(u_1,\ldots,u_N) \in X_{r, T}$
\begin{align*}
&\|\mathcal{T}_i(u)\|_{L^p(\mathbb R^d)}\leq \|\mathcal{I}_{i,1}\|_{L^p(\mathbb R^d)}+\|\mathcal{I}_{i,2}\|_{L^p(\mathbb R^d)},\\
&\mathcal{I}_{i,1}=G_\varepsilon\ast u_i^0, \quad \mathcal{I}_{i,2}=\int_0^tG_\varepsilon(.,t-s)\ast \left(\mathrm{div}(u_i \nabla V_i[u_1,\ldots,u_N])+g_i(u_1,\ldots,u_N)\right)(.,s) ds,\\
&\|\mathcal{I}_{i,1}\|_{L^p(\mathbb R^d)}\leq \|u_i^0\|_{L^p(\mathbb R^d)},\\
&\|\mathcal{I}_{i,2}\|_{L^p(\mathbb R^d)}\leq \int _0^t \|g_i(u_1,\ldots,u_N)(.,s)\|_{L^p(\mathbb R^d)} ds\\
&+\int_0^t \|\nabla G_\varepsilon(.,t-s)\|_{L^1(\mathbb R^d)}\|u_i\|_{L^p(\mathbb R^d)}\|\nabla V_i[u_1,\ldots,u_N]\|_{L^{\infty}(\mathbb R^d)}ds.
\end{align*}
We assumed, $u\in X_{r,T}$ which means that $u(x,t)$ is in a bounded subset of $\mathbb R^N$ for all $x\in \mathbb R^d$ and $t\in [0,T]$. On the other hand, $g$ is a $C^1$ function, so it is locally Lipschitz. Let $L$ be the Lipschitz constant of $g$ in the compact set that covers the image of $u$
\begin{align*}
    &|g(u_1,\ldots u_N)|\leq|g(u_1,\ldots u_N)-g(u_1^0,\ldots u_N^0)|+|g(u_1^0,\ldots u_N^0)| \leq L|u-u^0|+|g(u_1^0,\ldots u_N^0)|,\\
    &\|g(u_1,\ldots u_N)\|_{L^p(\mathbb R^d)}\leq L\|u\|_{L^p(\mathbb R^d)}+C,\\
    &\int _0^t \|g_i(u_1,\ldots,u_N)(.,s)\|_{L^p(\mathbb R^d)} ds\leq (Lr+C)t.
\end{align*}
Here, $C$ is a constant dependent on the initial data and the function $g$. The nonlocal term is bounded the same way as proved in \eqref{prop:local}
\[
\begin{gathered}
\|\mathcal{I}_{i,2}\|_{L^p(\mathbb R^d)}\leq (Lr+C_1)t + C_2r \int_0^t (t-s)^{-\frac{1}{2}}\|u_i(.,s)\|_{L^p(\mathbb R^d)} ds,\\
\vertiii{\mathcal{T}_i}_{T}\leq LrT+C_1T+C_2 r^2 T^{\frac{1}{2}},\\
\vertiii{\mathcal{T}}_{T}\leq (LrT+C_1T+C_2 r^2 T^{\frac{1}{2}})N.
\end{gathered}
\]
This operator is well-defined if there exists $T$ and $r$ such that the following inequality is satisfied
\[
(LrT+C_1T+C_2 r^2 T^{\frac{1}{2}})N\leq r.
\]
Now consider $u,\tilde{u}\in X_{r,T}$
\begin{align*}
    & \left\|\mathcal{T}_i(u)-\mathcal{T}_i(\tilde{u})\right\|_{L^p\left(\mathbb{R}^2\right)} \leq \int _0^t \|g_i(u_1,\ldots,u_N)(.,s)-g_i(\tilde u_1,\ldots,\tilde u_N)(.,s)\|_{L^p(\mathbb R^d)} ds\\
&+\int_0^t \|\nabla G_\varepsilon(.,t-s)\|_{L^1(\mathbb R^d)}\|u_i\nabla V_i[u_1,\ldots,u_N]-\tilde u_i\nabla V_i[\tilde u_1,\ldots,\tilde u_N]\|_{L^{p}(\mathbb R^d)}ds.
\end{align*}
We know that $g\in C^1$ so by the mean value theorem, we know there exists a $\bar u$ in between $u$ and $\tilde u$ such that:
\begin{align*}
     &\|g_i(u_1,\ldots,u_N)(.,s)-g_i(\tilde u_1,\ldots,\tilde u_N)(.,s)\|_{L^p(\mathbb R^d)} =\|g'_i(\bar{u}_1,\ldots,\bar{u}_N)(u-\tilde{u})\|_{L^p(\mathbb R^d)}\\
     &\leq \|g'_i(\bar{u}_1,\ldots,\bar{u}_N)\|_{L^\infty(\mathbb R^d)}\|u-\tilde{u}\|_{L^p(\mathbb R^d)} \leq C \|u-\tilde{u}\|_{L^p(\mathbb R^d)},
\end{align*}
where $C$ depends on the function $g$ and $r$.Also, we know that the following inequality holds true:
\begin{align*}
    &\|u_i\nabla V_i[u_1,\ldots,u_N]-\tilde u_i\nabla V_i[\tilde u_1,\ldots,\tilde u_N]\|_{L^{p}(\mathbb R^d)}\\
    &\leq\| u\|_{L^{p}(\mathbb R^d)}\left\| \nabla V_{i}\left[u_1,\dots,u_N\right]-\nabla V_{i}\left[\tilde{u} _1,\dots,\tilde{u}_N\right]\right\|_{L^{\infty}(\mathbb R^d)} +\left\|u-\tilde{u}\right\|_{L^{p}(\mathbb R^d)}\left\| \nabla V_{i}\left[\tilde{u}_1,\ldots,\tilde{u}_N\right] \right\|_{L^{\infty}(\mathbb R^d)}\\
    &\leq C_1r\left\|u-\tilde{u}\right\|_{L^{p}(\mathbb R^d)}+C_2 \sup_{0\leq t\leq T}\left\|u-\tilde{u}\right\|_{L^{1}(\mathbb R^d)}\| u\|_{L^{p}(\mathbb R^d)}.
\end{align*}
Eventually, we have that:
\begin{align*}
    & \left\|\mathcal{T}_1\left(u\right)-\mathcal{T}_1\left(\tilde u\right)\right\|_{L^{p}\left(\mathbb{R}^2\right)} \leq  Ct \vertiii{u-\tilde{u}}_T + C t^{\frac{1}{2}}\left(\vertiii{u-\tilde{u}}_T+r\left(\vertiii{u-\tilde{u}}_{T} \right)\right),\\
    &\vertiii{\mathcal{T}_1\left(u\right)-\mathcal{T}_1\left(\tilde u\right)}_T\leq  \left(Ct + C t^{\frac{1}{2}}\left(1+r \right)\right)\vertiii{u-\tilde{u}}_T.
\end{align*}
Hence, the operator is a contraction, if $C_1t+ C_2 t^{\frac{1}{2}}\left(1+r \right)< 1$. We can easily find $r$ and $T$ to satisfy both of the conditions derived, thus concluding the proof.
\end{proof}


\end{document}